\documentclass[11 pt, twoside]{amsart}
\usepackage{parskip}
\usepackage{amsthm,amsmath,amssymb,oldgerm}
\usepackage[a4paper, margin=2.5cm]{geometry}
\usepackage{mathrsfs}
\usepackage{verbatim}
\usepackage{hyperref}
\theoremstyle{plain}
\newtheorem{thm}{Theorem}[section]
\newtheorem*{theorem*}{Main Theorem}
\newtheorem{mainthm}[]{Main Theorem}[]
\newtheorem{lem}[thm]{Lemma}
\newtheorem{prop}[thm]{Proposition}
\newtheorem{cor}[thm]{Corollary}
\theoremstyle{definition}
\theoremstyle{definition}
\newtheorem{rem}[thm]{Remark}
\def\Bt{\mathbb{B}^{2}}

\def \G {\Gamma}

\def \C {\mathbb{C}}

\def \R {\mathbb{R}}

\def \calf {\mathcal{F}}

\def \calh {\mathcal{H}}

\def \cala {\mathcal{A}}

\def \cals {\mathcal{S}}

\def\det{\text{det}}

\DeclareMathOperator{\bsk} {{\it{\mathcal{B}}}_{\Gamma}^{{\it {k}}}}
\DeclareMathOperator{\bk} {{\it{\mathcal{B}}}_{\Gamma}^{3{\it {k}}}}
\DeclareMathOperator{\blk}{{\it \mathcal{B}}_{\Lambda}^{{ \it k}}}

\DeclareMathOperator{\ctildek}{{\it{\tilde{C}_{X_{\mathrm{\Gamma}},k}}}}
\DeclareMathOperator{\ctildethreek}{{\it{\tilde{C}_{X_{\mathrm{\Gamma}},\mathrm{3}k}}}}
\DeclareMathOperator{\sqctildethreek}{{\it{\tilde{C}^{\mathrm{2}}_{X_{\mathrm{\Gamma}},\mathrm{3}k}}}}
\DeclareMathOperator{\cuctildethreek}{{\it{\tilde{C}^{\mathrm{3}}_{X_{\mathrm{\Gamma}},\mathrm{3}k}}}}
\DeclareMathOperator{\cthreek}{{\it{C_{X_{\mathrm{\Gamma}},\mathrm{3}k}}}}
\DeclareMathOperator{\ck}{{\it{C_{X_{\mathrm{\Gamma}},k}}}}
\DeclareMathOperator{\dhyp}{ d_{\mathrm{hyp}}}
\DeclareMathOperator{\hyp}{\mu_{\mathrm{hyp}}} 
\DeclareMathOperator{\hypvol}{\mu_{\mathrm{hyp}}^{\mathrm{vol}}}
 
\DeclareMathOperator{\bermet} {\mu_{\mathrm{ber}}^{\it{k}}}
\DeclareMathOperator{\bervol} {\mu_{\mathrm{ber}}^{\it{k},\mathrm{vol}}}

\DeclareMathOperator{\hypbn}{\mu_{\mathrm{hyp}}}
\DeclareMathOperator{\hypbnvol}{\mu_{\mathrm{hyp}}^{\mathrm{vol}}}

\DeclareMathOperator{\rx}{{\it{r_{X_{\G}}}}}
\title[Off-diagonal estimates of Bergman kernels and Bergman metric]{Estimates of Bergman kernels and Bergman metric on compact Picard surfaces}
\author{Anilatmaja Aryasomayajula}
\address{Department of Mathematics, Indian Institute of Science Education and Research (IISER) Tirupati, 
Transit campus at Sri Rama Engineering College, Karkambadi Road,
Mangalam (B.O),Tirupati-517507, India.}
\email{anilatmaja@gmail.com}
\author{Dyuti Roy}
\address{Department of Mathematics, Indian Institute of Science Education and Research (IISER) Tirupati, 
Transit campus at Sri Rama Engineering College, Karkambadi Road,
Mangalam (B.O),Tirupati-517507, India.}
\email{dyutiroy@students.iisertirupati.ac.in}
\author{Debasish Sadhukhan}
\address{Department of Mathematics, Indian Institute of Science Education and Research (IISER) Tirupati, 
Transit campus at Sri Rama Engineering College, Karkambadi Road,
Mangalam (B.O),Tirupati-517507, India.}
\email{debasish@students.iisertirupati.ac.in}
\date{\today}
\subjclass[2010]{32A25, 32N10, 32N05, 53C07}
\keywords{Bergman metric, Ball quotients, Picard cusp forms}
\begin{document}
%%%%%%%%%%%%%%%%%%%%%%%%%%%%%%%%%%%%%%%%%%%%%%%%%%%%%%%%%%%%%%%%%%%%%%%%%%%%%%%%%%%%%%%%%%%%
\begin{abstract}
%%%%%%%%%%%%%%%%%%%%%%%%%%%%%%%%%%%%%%%%%%%%%%%%%%%%%%%%%%%%%%%%%%%%%%%%%%%%%%%%%%%%%%%%%%%%
Let $\G\subset \mathrm{SU}((2,1),\C)$ be a torsion-free cocompact subgroup.  Let $\Bt$ denote the $2$-dimensional complex ball endowed with the hyperbolic metric $\hyp$, and let $X_{\G}:=\G\backslash \Bt$ denote the quotient space, which is a compact complex manifold of dimension $2$. Let $\Lambda:= \Omega_{X_{\G}}^{2}$ denote the line bundle on $X_{\G}$, whose sections are holomorphic $(2,0)$-forms. For any $k\geq 1$, the hyperbolic metric induces a point-wise metric on $H^{0}(X_{\G},\Lambda^{\otimes k })$, which we denote by $|\cdot|_{\mathrm{hyp}}$. For any $k\geq 1$, let $\blk$ denote the Bergman kernel of the complex vector space $H^{0}(X_{\G},\Lambda^{\otimes k })$. For any $k\geq 3$, and $z,w\in X_{\G}$, the first main result of the article is the following off-diagonal estimate of the Bergman kernel $\blk$
\begin{align*}
\big|\blk(z,w)\big|_{\mathrm{hyp}}=O_{X_{\G}}\bigg(\frac{k^2}{\cosh^{3k-8}(\dhyp(z,w)\slash 2)}\bigg),
\end{align*}
%%%%%%%%%%%%%%%%%%%%%%%%%%%%%%%%%%%%%%%%%%%%%%%%%%%%%%%%%%%%%%%%%%%%%%%%%%%%%%%%%%%%%%%%%%%%
where $\dhyp(z,w)$ denotes the geodesic distance between the points $z$ and $w$  on $X_{\G}$, and the implied constant depends only on the Picard surface $X_{\G}$. 
%%%%%%%%%%%%%%%%%%%%%%%%%%%%%%%%%%%%%%%%%%%%%%%%%%%%%%%%%%%%%%%%%%%%%%%%%%%%%%%%%%%%%%%%%%%%

\vspace{0.1cm}\noindent
For any $k\geq 1$, let $\bermet(z):=-\frac{i}{2\pi}\partial_{z}\partial_{\overline{z}}\log|\blk(z,z)|_{\mathrm{hyp}}$ denote the Bergman metric associated the line bundle 
$\Lambda^{\otimes k}$, and let $\bervol$ denote the associated volume form.  For $k\gg 1$ sufficiently large, and $\epsilon>0$, the second main result of the article is the following estimate 
\begin{align*}
\sup_{z\in X_{\G}}\bigg|\frac{\bervol(z)}{\hypvol(z)}\bigg|=O_{X_{\G},\epsilon}\big(k^{4+\epsilon}\big),
\end{align*}
%%%%%%%%%%%%%%%%%%%%%%%%%%%%%%%%%%%%%%%%%%%%%%%%%%%%%%%%%%%%%%%%%%%%%%%%%%%%%%%%%%%%%%%%%%%%
where $\hypvol$ denotes the volume form associated to the hyperbolic metric $\hyp$, and the implied constant depends on the Picard surface $X_{\G}$, and on the choice of $\epsilon>0$.
\end{abstract}
%%%%%%%%%%%%%%%%%%%%%%%%%%%%%%%%%%%%%%%%%%%%%%%%%%%%%%%%%%%%%%%%%%%%%%%%%%%%%%%%%%%%%%%%%%%%
\maketitle
%%%%%%%%%%%%%%%%%%%%%%%%%%%%%%%%%%%%%%%%%%%%%%%%%%%%%%%%%%%%%%%%%%%%%%%%%%%%%%%%%%%%%%%%%%%%
\section{Introduction}\label{sec-1}
%%%%%%%%%%%%%%%%%%%%%%%%%%%%%%%%%%%%%%%%%%%%%%%%%%%%%%%%%%%%%%%%%%%%%%%%%%%%%%%%%%%%%%%%%%%%
In this section, we first discuss the history associated to the problem, and then describe the main results of the article. 
%%%%%%%%%%%%%%%%%%%%%%%%%%%%%%%%%%%%%%%%%%%%%%%%%%%%%%%%%%%%%%%%%%%%%%%%%%%%%%%%%%%%%%%%%%%%
%%%%%%%%%%%%%%%%%%%%%%%%%%%%%%%%%%%%%%%%%%%%%%%%%%%%%%%%%%%%%%%%%%%%%%%%%%%%%%%%%%%%%%%%%%%%
\subsection{History }\label{subsec-1.1}
%%%%%%%%%%%%%%%%%%%%%%%%%%%%%%%%%%%%%%%%%%%%%%%%%%%%%%%%%%%%%%%%%%%%%%%%%%%%%%%%%%%%%%%%%%%%
Estimates of Bergman kernels associated to high tensor-powers of holomorphic line bundles defined over complex manifolds is a topic of great interest in complex geometry.  Furthermore, comparison of different K\"ahler structures on a complex manifold is another related problem, and is also of intrigue in the subject. In this article, we address both these problems, in the setting of Picard surfaces. 
%%%%%%%%%%%%%%%%%%%%%%%%%%%%%%%%%%%%%%%%%%%%%%%%%%%%%%%%%%%%%%%%%%%%%%%%%%%%%%%%%%%%%%%%%%%%

\vspace{0.1cm}
Estimates of Bergman kernels associated to high tensor-powers of holomorphic line bundles on a compact complex manifold have been optimally derived by the likes of Tian, Zelditch, Demailly, and more recently by Ma and Marinescu. However, off-diagonal estimates of the Bergman kernels are difficult to obtain. Christ in \cite{christ}, Delin in \cite{delin}, Lu and Zelditch in \cite{luzel}, have derived off-diagonal estimates of high tensor-powers of Bergman kernels associated to holomorphic line bundles defined over a complex manifold. In \cite{ma}, Ma and Marinescu have also derived off-diagonal estimates of the Bergman kernel, and the estimates are the sharpest in the most general context of an arbitrary compact complex manifold. 
%%%%%%%%%%%%%%%%%%%%%%%%%%%%%%%%%%%%%%%%%%%%%%%%%%%%%%%%%%%%%%%%%%%%%%%%%%%%%%%%%%%%%%%%%%%%

\vspace{0.1cm}
Let $X$ be a compact hyperbolic Riemann surface, and let $\hyp$ denote the hyperbolic metric on $X$, the natural metric on $X$, compatible with its complex structure. Let $\Omega_{X}$ denote the cotangent bundle on $X$. For any $k\geq 1$, the hyperbolic metric on $X$ induces a point-wise metric on $H^{0}(X,\Omega_{X}^{\otimes k})$, which we denote by $|\cdot|_{\mathrm{hyp}}$. Let $\mathcal{B}_{\Omega_{X}}^{k}$ denote the Bergman kernel associated to the line bundle $\Omega_{X}^{\otimes k}$. In \cite{anil1} and \cite{anil2}, 
for any $k\geq 3$, and $z,w\in X_{\G}$, the following estimate is optimally derived
\begin{align}\label{am}
\big|\mathcal{B}_{\Omega_{X}}^{k}(z,w) \big|_{\mathrm{hyp}}=O_{X}\bigg(\frac{k}{\cosh^{k}\big(\dhyp(z,w)\slash 2\big)} \bigg),
\end{align} 
%%%%%%%%%%%%%%%%%%%%%%%%%%%%%%%%%%%%%%%%%%%%%%%%%%%%%%%%%%%%%%%%%%%%%%%%%%%%%%%%%%%%%%%%%%%%
where $\dhyp(z,w)$ denotes the geodesic distance between the points $z$ and $w$ on $X$, and the implied constant depends only on the Riemann surface $X$. 
%%%%%%%%%%%%%%%%%%%%%%%%%%%%%%%%%%%%%%%%%%%%%%%%%%%%%%%%%%%%%%%%%%%%%%%%%%%%%%%%%%%%%%%%%%%%

\vspace{0.1cm}
Estimate \eqref{am} is stronger than the estimates derived by Ma and Marinecsu in \cite{ma}, when restricted to the setting of a compact hyperbolic Riemann surface.
%%%%%%%%%%%%%%%%%%%%%%%%%%%%%%%%%%%%%%%%%%%%%%%%%%%%%%%%%%%%%%%%%%%%%%%%%%%%%%%%%%%%%%%%%%%%

\vspace{0.1cm}
We now describe the history associated to the second problem, namely that of comparing the Bergman metric associated to high tensor-power of a holomorphic line bundle defined over a complex manifold, with the natural metric on the complex manifold. We now describe the main result from \cite{anil-biswas}. With notation as above, let $X$ be a compact hyperbolic Riemann surface, and let $\Omega_{X}$ denote the cotangent bundle. With notation as above, for any $k\geq 1$, and $z\in X$, the Bergman metric associated to the line bundle $\Omega_{X}$ is given by the following equation 
\begin{align*}
\mu_{\mathrm{ber}}^k(z):=-\frac{i}{2\pi}\partial_{z}\partial_{\overline{z}}\log\big|\mathcal{B}_{\Omega_{X}}^{k}(z,z) \big|_{\mathrm{hyp}}.
\end{align*}
%%%%%%%%%%%%%%%%%%%%%%%%%%%%%%%%%%%%%%%%%%%%%%%%%%%%%%%%%%%%%%%%%%%%%%%%%%%%%%%%%%%%%%%%%%%%
For any $k\geq 3$, we have the following estimate from \cite{anil-biswas}
\begin{align}\label{ab}
\sup_{z\in X}\bigg|\frac{\mu_{\mathrm{ber}}^k(z)}{\hyp(z)} \bigg|=O_{X}(k^2),
\end{align}
%%%%%%%%%%%%%%%%%%%%%%%%%%%%%%%%%%%%%%%%%%%%%%%%%%%%%%%%%%%%%%%%%%%%%%%%%%%%%%%%%%%%%%%%%%%%
where the implied constant depends only on $X$. Estimate \cite{anil-biswas} is optimally derived. In \cite{anil-biswas}, using the above estimate, the authors compare two different K\"ahler metrics on $\mathrm{Sym}^{d}(X)$, the $d$-fold symmetric product of the compact hyperbolic Riemann surface $X$. 
%%%%%%%%%%%%%%%%%%%%%%%%%%%%%%%%%%%%%%%%%%%%%%%%%%%%%%%%%%%%%%%%%%%%%%%%%%%%%%%%%%%%%%%%%%%%

\vspace{0.1cm}
Estimate \eqref{ab} was extended to noncompact hyperbolic Riemann surfaces of finite volume in \cite{anil-arijit}. We refer the reader to \cite{abm} for a result of similar flavour. 
%%%%%%%%%%%%%%%%%%%%%%%%%%%%%%%%%%%%%%%%%%%%%%%%%%%%%%%%%%%%%%%%%%%%%%%%%%%%%%%%%%%%%%%%%%%%
%%%%%%%%%%%%%%%%%%%%%%%%%%%%%%%%%%%%%%%%%%%%%%%%%%%%%%%%%%%%%%%%%%%%%%%%%%%%%%%%%%%%%%%%%%%%

\vspace{0.2cm}
\subsection{Statements of main results}\label{subsec-1.1}
%%%%%%%%%%%%%%%%%%%%%%%%%%%%%%%%%%%%%%%%%%%%%%%%%%%%%%%%%%%%%%%%%%%%%%%%%%%%%%%%%%%%%%%%%%%%
We now describe the main results of the article. Let $\Bt\subset \C^2$ denote the $2$-dimensional unit complex ball, and let $\G\subset \mathrm{SU}((2,1),\C)$ be a torsion-free,  cocompact  subgroup. The special unitary group $\mathrm{SU}((2,1),\C)$ acts on $\Bt$ via fractional linear transformations, and let $X_{\G}:=\G\backslash \Bt$ denote the quotient space. 
%%%%%%%%%%%%%%%%%%%%%%%%%%%%%%%%%%%%%%%%%%%%%%%%%%%%%%%%%%%%%%%%%%%%%%%%%%%%%%%%%%%%%%%%%%%%

\vspace{0.1cm}
The quotient space $X_{\G}$ admits the structure of a compact complex manifold of dimension $2$, and let $\hyp$ denote the hyperbolic metric, the natural metric on $X_{\G}$, 
which is compatible with the complex structure of $X_{\G}$. Let $\Lambda:=\Omega^2_{X_{\G}}$ denote the line bundle on $X_{\G}$, whose sections are holomorphic $(2,0)$-forms. 
%%%%%%%%%%%%%%%%%%%%%%%%%%%%%%%%%%%%%%%%%%%%%%%%%%%%%%%%%%%%%%%%%%%%%%%%%%%%%%%%%%%%%%%%%%%%

\vspace{0.1cm}
For any $k\geq 1$, let $H^{0}( X_{\G},\Lambda^{\otimes k})$ denote the space of global sections of the line bundle $\Lambda^{\otimes k}$. The hyperbolic metric induces a point-wise metric $|\cdot|_{\mathrm{hyp}}$, and an $L^2$ inner-product $\langle\cdot,\cdot\rangle_{\mathrm{hyp}}$ on the complex vector space $H^{0}( X_{\G},\Lambda^{\otimes k})$. 
Let $\lbrace \omega_1,\ldots,\omega_{d_{k}}\rbrace$ be an orthonormal basis of $H^{0}( X_{\G},\Lambda^{\otimes k})$, with respect to the $L^2$ inner-product $\langle\cdot,\cdot\rangle_{\mathrm{hyp}}$, where $d_k$ is the dimension of $H^{0}( X_{\G},\Lambda^{\otimes k})$. 
%%%%%%%%%%%%%%%%%%%%%%%%%%%%%%%%%%%%%%%%%%%%%%%%%%%%%%%%%%%%%%%%%%%%%%%%%%%%%%%%%%%%%%%%%%%%

\vspace{0.1cm}
For any $z,w\in \Bt$, the Bergman kernel associated to the line bundle $\Lambda^{\otimes k}$ is given by the following formula
\begin{align*}
\blk(z,w):=\sum_{j=1}^{d_{k}}\omega_{j}(z)\overline{\omega_{j}(w)}.
\end{align*}
%%%%%%%%%%%%%%%%%%%%%%%%%%%%%%%%%%%%%%%%%%%%%%%%%%%%%%%%%%%%%%%%%%%%%%%%%%%%%%%%%%%%%%%%%%%%
From Riesz representation theorem, it follows that the definition of $\blk$ is independent of the choice of orthonormal basis. 
%%%%%%%%%%%%%%%%%%%%%%%%%%%%%%%%%%%%%%%%%%%%%%%%%%%%%%%%%%%%%%%%%%%%%%%%%%%%%%%%%%%%%%%%%%%%

\vspace{0.1cm}
For any $z\in X_{\G}$, the Bergman metric associated to $\Lambda^{\otimes k}$ is given by the following formula
\begin{align*}
\bermet(z):=-\frac{i}{2\pi}\partial_{z}\partial_{\overline {z}}\big| \blk(z,z)\big|_{\mathrm{hyp}}.
\end{align*}
%%%%%%%%%%%%%%%%%%%%%%%%%%%%%%%%%%%%%%%%%%%%%%%%%%%%%%%%%%%%%%%%%%%%%%%%%%%%%%%%%%%%%%%%%%%%

\vspace{0.1cm}
The first main result is the following theorem, which is proved as Theorem \ref{thm4}.
%%%%%%%%%%%%%%%%%%%%%%%%%%%%%%%%%%%%%%%%%%%%%%%%%%%%%%%%%%%%%%%%%%%%%%%%%%%%%%%%%%%%%%%%%%%%

\vspace{0.1cm}
\begin{mainthm}\label{mainthm1}
With notation as above, for any $k\geq 3$, and $z,w\in X_{\G}$, we have the following estimate
\begin{align}\label{mainthm1-est}
\big| \blk(z,w)\big|_{\mathrm{hyp}}=O_{X_{\G}}\bigg(\frac{k^{2}}{\cosh^{3k-8}\big(\dhyp(z,w)\slash 2\big)}\bigg),
\end{align}
%%%%%%%%%%%%%%%%%%%%%%%%%%%%%%%%%%%%%%%%%%%%%%%%%%%%%%%%%%%%%%%%%%%%%%%%%%%%%%%%%%%%%%%%%%%%
where $\dhyp(z,w)$ denotes the geodesic distance between $z$ and $w$, and the implied constant depends only on the Picard surface $X_{\G}$.
\end{mainthm}
%%%%%%%%%%%%%%%%%%%%%%%%%%%%%%%%%%%%%%%%%%%%%%%%%%%%%%%%%%%%%%%%%%%%%%%%%%%%%%%%%%%%%%%%%%%%

\vspace{0.1cm}
The second main result is the following theorem, which is proved as Corollary \ref{cor10}.
%%%%%%%%%%%%%%%%%%%%%%%%%%%%%%%%%%%%%%%%%%%%%%%%%%%%%%%%%%%%%%%%%%%%%%%%%%%%%%%%%%%%%%%%%%%%

\vspace{0.1cm}
\begin{mainthm}\label{mainthm2}
With notation as above, for $k\gg 1$ sufficiently large, and $\epsilon>0$, we have the following estimate
\begin{align}\label{mainthm2-est}
\sup_{z\in X_{\G}}\bigg|\frac{\bervol(z)}{\hypvol(z)} \bigg|=O_{X_{\G},\epsilon}\big(k^{4+\epsilon}\big),
\end{align}
%%%%%%%%%%%%%%%%%%%%%%%%%%%%%%%%%%%%%%%%%%%%%%%%%%%%%%%%%%%%%%%%%%%%%%%%%%%%%%%%%%%%%%%%%%%%
where $\bervol$ and $\hypvol$ denote the volume forms associated to the metrics $\bermet$ and $\hyp$, respectively, and the implied constant depends on the Picard surface $X_{\G}$, and on the choice of $\epsilon>0$.
\end{mainthm}
%%%%%%%%%%%%%%%%%%%%%%%%%%%%%%%%%%%%%%%%%%%%%%%%%%%%%%%%%%%%%%%%%%%%%%%%%%%%%%%%%%%%%%%%%%%%

\vspace{0.1cm}
\begin{rem}
Estimate \eqref{mainthm1-est} can be easily extended to higher dimensions, namely, Picard varieties. However, for the brevity of the exposition, we restrict ourselves to Picard surfaces. 
%%%%%%%%%%%%%%%%%%%%%%%%%%%%%%%%%%%%%%%%%%%%%%%%%%%%%%%%%%%%%%%%%%%%%%%%%%%%%%%%%%%%%%%%%%%%

\vspace{0.1cm}
We believe that estimate \eqref{mainthm2-est} can be improved by a factor of $k^{\epsilon}$, i.e., we believe that the following estimate is optimal
\begin{align*}
\sup_{z\in X_{\G}}\bigg|\frac{\bervol(z)}{\hypvol(z)} \bigg|=O_{X_{\G}}\big(k^{4}\big).
\end{align*}
\end{rem}
%%%%%%%%%%%%%%%%%%%%%%%%%%%%%%%%%%%%%%%%%%%%%%%%%%%%%%%%%%%%%%%%%%%%%%%%%%%%%%%%%%%%%%%%%%%%%%%%%%%%%%%%%%%%%%%%%%%%%%%%%%%%%%%%%%%%%%%%%%%%%%%%%%%%%%%%%%%%%%%%%%%%%%%%%%%%%%%%%%%%%%%%

\vspace{0.2cm}
\section{Preliminaries}\label{sec2}
%%%%%%%%%%%%%%%%%%%%%%%%%%%%%%%%%%%%%%%%%%%%%%%%%%%%%%%%%%%%%%%%%%%%%%%%%%%%%%%%%%%%%%%%%%%%
In this section, we introduce the basic notions, and set up the notation, which are essential for proving Main Theorem \ref{mainthm1} and Main Theorem \ref{mainthm2}.
%%%%%%%%%%%%%%%%%%%%%%%%%%%%%%%%%%%%%%%%%%%%%%%%%%%%%%%%%%%%%%%%%%%%%%%%%%%%%%%%%%%%%%%%%%%%
%%%%%%%%%%%%%%%%%%%%%%%%%%%%%%%%%%%%%%%%%%%%%%%%%%%%%%%%%%%%%%%%%%%%%%%%%%%%%%%%%%%%%%%%%%%%

\vspace{0.1cm}
\subsection{Hyperbolic Space}\label{sec2.1}
%%%%%%%%%%%%%%%%%%%%%%%%%%%%%%%%%%%%%%%%%%%%%%%%%%%%%%%%%%%%%%%%%%%%%%%%%%%%%%%%%%%%%%%%%%%%
Let $\calh^{2}(\C)$ denote the $2$-dimensional complex hyperbolic space, which can also be identified with the unit ball $\Bt$, which is given by the following equation 
\begin{align*}
\Bt :=\big\lbrace z:=\big(z_{1},z_{2}\big)^{t}\in\C^{2}\big|\,|z|^{2}:=|z_{1}|^{2}+|z_{2}|^{2}<1\big\rbrace.
\end{align*}
%%%%%%%%%%%%%%%%%%%%%%%%%%%%%%%%%%%%%%%%%%%%%%%%%%%%%%%%%%%%%%%%%%%%%%%%%%%%%%%%%%%%%%%%%%%%

\vspace{0.1cm}
We now describe the identification of $\calh^{2}(\C)$ with $\Bt$. Let  $H$  be a Hermitian matrix of signature $(2,1)$, and let $(\mathbb{C}^{3},\langle \cdot , \cdot\rangle_\mathrm{hyp})$ denote the Hermitian inner-product space $\C^{3}$ equipped with the Hermitian inner-product $\langle\cdot, \cdot\rangle_{\mathrm{hyp}}$, induced by $H$. For $z,w \in \mathbb{C}^{3}$, the Hermitian inner-product $\langle\cdot, \cdot\rangle_{\mathrm{hyp}}$ is given by the following equation
\begin{align*}
\langle z, w \rangle_{\mathrm{hyp}} = w^{*}H z,
\end{align*}
%%%%%%%%%%%%%%%%%%%%%%%%%%%%%%%%%%%%%%%%%%%%%%%%%%%%%%%%%%%%%%%%%%%%%%%%%%%%%%%%%%%%%%%%%%%%
where $w^{*}$ denotes the complex conjugate transpose of $w$.
%%%%%%%%%%%%%%%%%%%%%%%%%%%%%%%%%%%%%%%%%%%%%%%%%%%%%%%%%%%%%%%%%%%%%%%%%%%%%%%%%%%%%%%%%%%%

\vspace{0.1cm}
Set
\begin{align*}
&V_{-}(H):= \big\lbrace z\in\mathbb{C}^{3}\big|\,\langle z, z \rangle_{\mathrm{hyp}} < 0 \big\rbrace;\\
  &V_0(H):= \{z\in\mathbb{C}^{3}\big|\,\langle z, z \rangle_{\mathrm{hyp}} = 0 \}\subset \C^{3};
\end{align*}
%%%%%%%%%%%%%%%%%%%%%%%%%%%%%%%%%%%%%%%%%%%%%%%%%%%%%%%%%%%%%%%%%%%%%%%%%%%%%%%%%%%%%%%%%%%%
and
\begin{align*}
\cala^{3}:=\big\lbrace \big(z_{1},z_{2},z_{3}\big)^{t}\in \C^{3}\big|\,z_{3}\not=0\rbrace\subset \C^{3}.
\end{align*}
%%%%%%%%%%%%%%%%%%%%%%%%%%%%%%%%%%%%%%%%%%%%%%%%%%%%%%%%%%%%%%%%%%%%%%%%%%%%%%%%%%%%%%%%%%%%

The subspace $\cala^{3}$ inherits the structure of a Hermitian inner-product space from $(\mathbb{C}^{3},\langle \cdot , \cdot\rangle_\mathrm{hyp})$, which we denote by $(\cala^{3},\langle \cdot , \cdot\rangle_{\mathrm{hyp}})$. We have the following map
\begin{align*}
&\phi:\cala_{3} \longrightarrow \C^{2}\\&
\phi\big((z_1,z_2,z_{3})^{t}\big)= \bigg( \frac{z_1}{z_{3}},  \frac{z_2}{z_{3}} \bigg)^{t}.  
\end{align*}                    
%%%%%%%%%%%%%%%%%%%%%%%%%%%%%%%%%%%%%%%%%%%%%%%%%%%%%%%%%%%%%%%%%%%%%%%%%%%%%%%%%%%%%%%%%%%% 

The complex hyperbolic space  $\mathcal{H}^{2}(\mathbb{C})$  and its boundary $\partial \mathcal{H}^{2}(\mathbb{C})$ are given by the following equations
\begin{align*}
\mathcal{H}^{2}\left(\mathbb{C}\right):=\phi\big( V_{-}(H)\big)\subset \C^{2};\,\,\,\partial \mathcal{H}^{2}(\mathbb{C}):=\phi\big( V_{0}(H)\big)\subset\C^{2}.
\end{align*}
%%%%%%%%%%%%%%%%%%%%%%%%%%%%%%%%%%%%%%%%%%%%%%%%%%%%%%%%%%%%%%%%%%%%%%%%%%%%%%%%%%%%%%%%%%%%

\vspace{0.1cm}
For any $z= (z_1 ,z_2)^{t} \in \mathcal{H}^2(\C)$, let $\tilde{z}:=(z_1,z_2,1)^{t}$ denote the lift of $z$ to $\cala^{3}$. For 
\begin{align*}
H = \left(\begin{smallmatrix}
1 & 0 & 0\\
0 & 1 & 0 \\
0  & 0 & -1
\end{smallmatrix}\right),
\end{align*}
%%%%%%%%%%%%%%%%%%%%%%%%%%%%%%%%%%%%%%%%%%%%%%%%%%%%%%%%%%%%%%%%%%%%%%%%%%%%%%%%%%%%%%%%%%%%
observe that
\begin{align*}
\langle \tilde{z},\tilde{z}\rangle_{\mathrm{hyp}} =\tilde{z}^{\ast}H\tilde{z}= |z_1|^2 + |z_2|^2-1< 0\implies |z_1|^2 + |z_2|^2 <1.
\end{align*}
%%%%%%%%%%%%%%%%%%%%%%%%%%%%%%%%%%%%%%%%%%%%%%%%%%%%%%%%%%%%%%%%%%%%%%%%%%%%%%%%%%%%%%%%%%%%

Thus $\mathcal{H}^2(\C)$ can be identified with the open unit ball 
\begin{align*}
\mathbb{B}^2 := \big\lbrace z:=(z_{1},z_{2})^{t}\in \C^{2}\big|\, |z|^{2}:=|z_1|^2+ |z_2|^2 < 1 \big\rbrace,
\end{align*}
%%%%%%%%%%%%%%%%%%%%%%%%%%%%%%%%%%%%%%%%%%%%%%%%%%%%%%%%%%%%%%%%%%%%%%%%%%%%%%%%%%%%%%%%%%%%
which is known as the ball model of the complex hyperbolic space $\calh^2(\C)$. 
%%%%%%%%%%%%%%%%%%%%%%%%%%%%%%%%%%%%%%%%%%%%%%%%%%%%%%%%%%%%%%%%%%%%%%%%%%%%%%%%%%%%%%%%%%%%

\vspace{0.1cm}
Thus, for any $z:=(z_1,z_2)^t,\,w:=(w_1,w_2)^t\in \Bt$ with corresponding lifts  $\tilde{z}:=(z_1,z_2,1)^t,\,\tilde{w}:=(w_1,w_2,1)^t\in \cala^{3}$, respectively, we have 
\begin{align}\label{hyp-ip}
\langle \tilde{z},\tilde{w}\rangle_{\mathrm{hyp}}:=z_1\overline{w_1}+z_2\overline{w_2}-1.
\end{align} 
%%%%%%%%%%%%%%%%%%%%%%%%%%%%%%%%%%%%%%%%%%%%%%%%%%%%%%%%%%%%%%%%%%%%%%%%%%%%%%%%%%%%%%%%%%%%

The hyperbolic ball $\Bt$ is equipped with the hyperbolic metric $\hypbn$, which has constant negative curvature $-1$. For any given $z=(z_{1},z_{2})^{t}\in \Bt$, the hyperbolic metric 
$\hypbn$ is given by the following formula
\begin{align}\label{defnhypmetric}
\mu_{\mathrm{hyp}}(z):=-\frac{i}{2}\partial\overline{\partial}\log\big(1-|z|^{2}\big),
\end{align}
%%%%%%%%%%%%%%%%%%%%%%%%%%%%%%%%%%%%%%%%%%%%%%%%%%%%%%%%%%%%%%%%%%%%%%%%%%%%%%%%%%%%%%%%%%%%
and the associated volume form $\hypbnvol$ is given by 
\begin{align}\label{defnhypvol}
\mu_{\mathrm{hyp}}^{\mathrm{vol}}(z) := \bigg(\frac{i}{2}\bigg)^{2}\frac{dz_1 \wedge d \overline{z_1}\wedge  dz_2 \wedge d \overline{z_2}}{\big(1-|z|^2\big)^{3}}.   
\end{align}
%%%%%%%%%%%%%%%%%%%%%%%%%%%%%%%%%%%%%%%%%%%%%%%%%%%%%%%%%%%%%%%%%%%%%%%%%%%%%%%%%%%%%%%%%%%%

For any $z,w\in \Bt$ with respective lifts $\tilde{z},\tilde{w}\in\cala^3$, the hyperbolic distance $\dhyp(z,w)$, which is determined by the hyperbolic metric $\hyp$, is given by the following formula
\begin{align}\label{defn:hyp}
\cosh^2\big(\dhyp(z,w)\slash 2\big):=\frac{\langle\tilde{z},\tilde{w} \rangle_{\mathrm{hyp}} \langle\tilde{w},\tilde{z} \rangle_{\mathrm{hyp}}}{\langle\tilde{z}, \tilde{z}\rangle_{\mathrm{hyp}} \langle\tilde{w}, \tilde{w}\rangle_{\mathrm{hyp}}}.
\end{align}
%%%%%%%%%%%%%%%%%%%%%%%%%%%%%%%%%%%%%%%%%%%%%%%%%%%%%%%%%%%%%%%%%%%%%%%%%%%%%%%%%%%%%%%%%%%%
%%%%%%%%%%%%%%%%%%%%%%%%%%%%%%%%%%%%%%%%%%%%%%%%%%%%%%%%%%%%%%%%%%%%%%%%%%%%%%%%%%%%%%%%%%%%

\vspace{0.2cm}
\subsection{Picard Surface}\label{subsec-2.2}
%%%%%%%%%%%%%%%%%%%%%%%%%%%%%%%%%%%%%%%%%%%%%%%%%%%%%%%%%%%%%%%%%%%%%%%%%%%%%%%%%%%%%%%%%%%%
The unitary group $\mathrm{SU}((2,1),\C)$ is defined by 
\begin{align}\label{groupdfn}
\mathrm{SU}((2,1),\C):= \{ A \in \mathrm{SL}(3,\C) \; | \, A^*H A =  H\},\, \,\mathrm{where} \,\,
H = \left(\begin{smallmatrix}
1 & 0 & 0 \\
0 & 1 & 0 \\
0 & 0  & -1
\end{smallmatrix}\right).
\end{align}
%%%%%%%%%%%%%%%%%%%%%%%%%%%%%%%%%%%%%%%%%%%%%%%%%%%%%%%%%%%%%%%%%%%%%%%%%%%%%%%%%%%%%%%%%%%%
 
The group $ \mathrm{SU}((2,1),\C)$ acts transitively on $\mathbb{B}^2$ by fractional linear transformations. Any element $\gamma \in \mathrm{SU}((2,1),\C)$ can be written in terms of block matrices as 
%%%%%%%%%%%%%%%%%%%%%%%%%%%%%%%%%%%%%%%%%%%%%%%%%%%%%%%%%%%%%%%%%%%%%%%%%%%%%%%%%%%%%%%%%%%%
\begin{align*}
\gamma=\left(\begin{smallmatrix} A&B\\C &D\end{smallmatrix}\right) \in \Gamma,
\end{align*}
%%%%%%%%%%%%%%%%%%%%%%%%%%%%%%%%%%%%%%%%%%%%%%%%%%%%%%%%%%%%%%%%%%%%%%%%%%%%%%%%%%%%%%%%%%%%
where $A\in M_{2\times 2}(\C)$, $B\in M_{2\times 1}(\C)$, $C\in M_{1\times 2}(\C) $, and $D\in\C$. For $\gamma \in \mathrm{SU}((2,1),\C)$ and $z=(z_1,z_2)^{t}\in\Bt$, the action is defined as 
\begin{align}\label{gammact}
\gamma z:=\frac{Az+B}{Cz+D}.
\end{align}
%%%%%%%%%%%%%%%%%%%%%%%%%%%%%%%%%%%%%%%%%%%%%%%%%%%%%%%%%%%%%%%%%%%%%%%%%%%%%%%%%%%%%%%%%%%%

Here, $z$ is seen as a column vector vector in $\C^2$. 
%%%%%%%%%%%%%%%%%%%%%%%%%%%%%%%%%%%%%%%%%%%%%%%%%%%%%%%%%%%%%%%%%%%%%%%%%%%%%%%%%%%%%%%%%%%%

\vspace{0.1cm} 
Let $\Gamma\subset \mathrm{SU}((2,1),\C)$ be a discrete, torision-free, cocompact subgroup. Then $\Gamma$ also acts on $\Bt$ via fractional linear transformations, which are  as described in equation \eqref{gammact}. The quotient space $X_{\G}:=\Gamma\backslash\mathbb{B}^2$ admits the structure of a compact complex manifold of dimension $2$, and is called a compact Picard surface.
%%%%%%%%%%%%%%%%%%%%%%%%%%%%%%%%%%%%%%%%%%%%%%%%%%%%%%%%%%%%%%%%%%%%%%%%%%%%%%%%%%%%%%%%%%%%

\vspace{0.1cm}
The hyperbolic metric $\hypbn$ is preserved under the action of $\mathrm{SU}((2,1),\C)$, and defines a metric on the quotient space $X_{\G}$, which we again denote by 
$\hypbn$, and is locally represented by the formula described in equation \eqref{defnhypmetric}.The volume form associated to $\hypbn$ is again denoted by $\hypbnvol$, and is as described 
by formula \eqref{defnhypvol}. 
%%%%%%%%%%%%%%%%%%%%%%%%%%%%%%%%%%%%%%%%%%%%%%%%%%%%%%%%%%%%%%%%%%%%%%%%%%%%%%%%%%%%%%%%%%%%

\vspace{0.1cm}
The hyperbolic distance function determined by the hyperbolic metric $\hypbn$ on $X_{\G}$ is again denoted by $\dhyp$. Locally, for any $z,w\in X_{\G}$, the geodesic distance 
$\dhyp(z,w)$ between the points $z$ and $w$ on $X_{\G}$ is given by formula \eqref{defn:hyp}. 
%%%%%%%%%%%%%%%%%%%%%%%%%%%%%%%%%%%%%%%%%%%%%%%%%%%%%%%%%%%%%%%%%%%%%%%%%%%%%%%%%%%%%%%%%%%%

\vspace{0.1cm}
The injectivity radius of $X_{\G}$ is given by the following formula 
\begin{align}\label{ircom}
r_{X_{\G}} := \inf \big\lbrace \dhyp(z,\gamma z)\big|\,z \in \Bt, \gamma\in\Gamma\backslash\lbrace\mathrm{Id}\rbrace\big\rbrace. 
\end{align}
%%%%%%%%%%%%%%%%%%%%%%%%%%%%%%%%%%%%%%%%%%%%%%%%%%%%%%%%%%%%%%%%%%%%%%%%%%%%%%%%%%%%%%%%%%%%

\vspace{0.2cm}
\subsection{Picard Cusp Forms}\label{sec-2.3}
%%%%%%%%%%%%%%%%%%%%%%%%%%%%%%%%%%%%%%%%%%%%%%%%%%%%%%%%%%%%%%%%%%%%%%%%%%%%%%%%%%%%%%%%%%%%
For any $k\geq 1$, a holomorphic function $f: \Bt \to \C$ is called a Picard cusp form of weight-$k$, with respect to $\G$, if for any $\gamma \in \Gamma$ and $z \in \Bt$, $f$ satisfies the following transformation property:
\begin{align}\label{cuspform:defn}
f\big( \gamma z\big) =\left({Cz+D}\right)^{k} f(z),\,\,\,
\mathrm{where}\,\,
 \gamma=\left(\begin{smallmatrix} A&B\\C &D\end{smallmatrix}\right) \in \Gamma,
\end{align}
%%%%%%%%%%%%%%%%%%%%%%%%%%%%%%%%%%%%%%%%%%%%%%%%%%%%%%%%%%%%%%%%%%%%%%%%%%%%%%%%%%%%%%%%%%%%
and $A\in M_{2\times 2}(\C)$, $B\in M_{2\times 1}(\C)$, $C\in M_{1\times 2}(\C) $, and $D\in\C$.
%%%%%%%%%%%%%%%%%%%%%%%%%%%%%%%%%%%%%%%%%%%%%%%%%%%%%%%%%%%%%%%%%%%%%%%%%%%%%%%%%%%%%%%%%%%%

\vspace{0.1cm}
The complex vector-space of weight-$k$ Picard cusp forms is denoted by $\cals_{k}(\Gamma)$.
%%%%%%%%%%%%%%%%%%%%%%%%%%%%%%%%%%%%%%%%%%%%%%%%%%%%%%%%%%%%%%%%%%%%%%%%%%%%%%%%%%%%%%%%%%%%

\vspace{0.1cm}
For any $f\in \cals_{k}(\Gamma)$, the Petersson norm at a point $z\in X_{\G}$ is given by the following formula
\begin{align}
\big|f(z)\big|_{\mathrm{pet}}^{2}:= \big( 1- |z|^2\big)^{k}|f(z)|^{2}.
\end{align}
%%%%%%%%%%%%%%%%%%%%%%%%%%%%%%%%%%%%%%%%%%%%%%%%%%%%%%%%%%%%%%%%%%%%%%%%%%%%%%%%%%%%%%%%%%%%

\vspace{0.1cm}
The Petersson norm induces an $L^2$-metric on $\cals_{k}(\Gamma)$, which we denote by $\langle\cdot,\cdot\rangle_{\mathrm{pet}}$, and is known as the Petersson inner-product. For any $f,g\in\cals_{k}(\G)$, the Petersson inner-product is given by the following formula
\begin{align*}
\langle f, g\rangle_{\mathrm{pet}}:=\int_{\calf_{\G}} \big( 1- |z|^2\big)^{k}f(z)\overline{g(z)}\hypbnvol(z),
\end{align*}
%%%%%%%%%%%%%%%%%%%%%%%%%%%%%%%%%%%%%%%%%%%%%%%%%%%%%%%%%%%%%%%%%%%%%%%%%%%%%%%%%%%%%%%%%%%%
where $\calf_{\G}$ denotes a fundamental domain of $X_{\G}$.
%%%%%%%%%%%%%%%%%%%%%%%%%%%%%%%%%%%%%%%%%%%%%%%%%%%%%%%%%%%%%%%%%%%%%%%%%%%%%%%%%%%%%%%%%%%%
%%%%%%%%%%%%%%%%%%%%%%%%%%%%%%%%%%%%%%%%%%%%%%%%%%%%%%%%%%%%%%%%%%%%%%%%%%%%%%%%%%%%%%%%%%%%

\vspace{0.2cm}
\subsection{Line bundle of differential forms}\label{subsec-2.3}
%%%%%%%%%%%%%%%%%%%%%%%%%%%%%%%%%%%%%%%%%%%%%%%%%%%%%%%%%%%%%%%%%%%%%%%%%%%%%%%%%%%%%%%%%%%%
Let $\Lambda:=\Omega_{X_{\G}}^{2}$ denote the sheaf of holomorphic differential $(2,0)$-forms. For any $k\geq 1$, we now establish the isometry between $H^{0}(X_{\G},\Lambda^{\otimes k})$, the space of global sections of $\Lambda^{\otimes k}$, and $\cals^{3k}(\Gamma)$, the complex vector space of weight-$3k$ Picard cusp forms, with respect to $\G$. 
%%%%%%%%%%%%%%%%%%%%%%%%%%%%%%%%%%%%%%%%%%%%%%%%%%%%%%%%%%%%%%%%%%%%%%%%%%%%%%%%%%%%%%%%%%%%

\vspace{0.1cm}
Recall that any $\gamma \in \Gamma$ can be expressed as
\begin{align*}
\gamma=\left(\begin{smallmatrix} A&B\\C &D\end{smallmatrix}\right) \in \Gamma,
\end{align*}
%%%%%%%%%%%%%%%%%%%%%%%%%%%%%%%%%%%%%%%%%%%%%%%%%%%%%%%%%%%%%%%%%%%%%%%%%%%%%%%%%%%%%%%%%%%%
where  
\begin{align*}
A= \left(\begin{smallmatrix} a_{11}&  a_{12} \\ a_{21} & a_{22}  \end{smallmatrix}\right)\in M_{2\times 2}(\C) ; \;\; B= (b_1,b_2)^{t}\in M_{2\times 1}(\C) ; \;\; C= (c_1,c_2) \in M_{1\times 2}(\C); \;\; \mathrm{and} \;\;D \in \C.
\end{align*}
%%%%%%%%%%%%%%%%%%%%%%%%%%%%%%%%%%%%%%%%%%%%%%%%%%%%%%%%%%%%%%%%%%%%%%%%%%%%%%%%%%%%%%%%%%%%

For any $z:=(z_1, z_2 )^{t}\in \Bt$, we have
\begin{align*}
\gamma z = \left(\frac{a_{11}z_1+a_{12}z_2 + b_1 }{c_1z_1 +c_2+D}, \frac{a_{21}z_1+a_{22}z_2 + b_2 }{c_1z_1 +c_2+D} \right)^{t} = \left( \gamma z_1, \gamma z_2 \right)^{t}.
\end{align*}
%%%%%%%%%%%%%%%%%%%%%%%%%%%%%%%%%%%%%%%%%%%%%%%%%%%%%%%%%%%%%%%%%%%%%%%%%%%%%%%%%%%%%%%%%%%%

The derivatives of $\gamma z_1$ and $\gamma z_2$ with respect to $z_1$ and $z_2$ are calculated below
\begin{align*}
 \frac{ \partial}{\partial z_1} \left( \gamma z_1 \right) = \frac{(a_{11}c_2 - a_{12}c_1 ) z_2 + \left( a_{11}d- b_1 c_1 \right)}{\left(c_1z_1+ c_2 z_2 + D \right)^2} ,\\ 
\frac{ \partial}{\partial z_2} \left( \gamma z_1 \right) = \frac{\left(a_{12}c_1 - a_{11}c_2 \right) z_1 + \left( a_{12}d- b_1 c_2 \right)}{\left(c_1z_1+ c_2 z_2 + D \right)^2} ; 
\end{align*}
%%%%%%%%%%%%%%%%%%%%%%%%%%%%%%%%%%%%%%%%%%%%%%%%%%%%%%%%%%%%%%%%%%%%%%%%%%%%%%%%%%%%%%%%%%%%
and
\begin{align*}
\frac{ \partial}{\partial z_1} \left( \gamma z_2 \right) = \frac{\left(a_{21}c_2 - a_{22}c_1 \right) z_2 + \left( a_{21}d- b_2 c_1 \right)}{\left(c_1z_1+ c_2 z_2 + D \right)^2}, \\
\frac{ \partial}{\partial z_2} \left( \gamma z_2 \right) = \frac{\left(a_{22}c_1 - a_{21}c_2 \right) z_1 + \left( a_{22}d- b_2 c_2 \right)}{\left(c_1z_1+ c_2 z_2 + D \right)^2}.
 \end{align*}
%%%%%%%%%%%%%%%%%%%%%%%%%%%%%%%%%%%%%%%%%%%%%%%%%%%%%%%%%%%%%%%%%%%%%%%%%%%%%%%%%%%%%%%%%%%%

From the above calculations, we deduce that  
 \begin{align} \label{dergam1}
& d(\gamma z_1) =  \frac{ \partial}{\partial z_1} ( \gamma z_1 ) dz_1 + \frac{ \partial}{\partial z_2} ( \gamma z_1) dz_2=\notag \\[0.1cm] 
 &  \frac{ \big( \left(a_{11}c_2 - a_{12}c_1 \right) z_2 + \left( a_{11}d- b_1 c_1 \right) \big)dz_1 + \big(\left(a_{12}c_1 - a_{11}c_2 \right) z_1 + \left( a_{12}d- b_1 c_2 \right) \big) dz_2}{\left(c_1z_1 + c_2 z_2 +D \right)^2}; 
   \end{align}
%%%%%%%%%%%%%%%%%%%%%%%%%%%%%%%%%%%%%%%%%%%%%%%%%%%%%%%%%%%%%%%%%%%%%%%%%%%%%%%%%%%%%%%%%%%%
and
\begin{align}\label{dergam2}
 &d(\gamma z_2) =  \frac{ \partial}{\partial z_1} \left( \gamma z_2 \right) dz_1 + \frac{ \partial}{\partial z_2} \left( \gamma z_2 \right) dz_2= \notag\\[0.1cm] 
  & \frac{ \left( \left(a_{21}c_2 - a_{22}c_1 \right) z_2 + \left( a_{21}d- b_2 c_1 \right) \right)dz_1 + \left (\left(a_{22}c_1 - a_{21}c_2 \right) z_1 + \left( a_{22}d- b_2 c_2 \right) \right) dz_2}{\left(c_1z_1 + c_2 z_2 +D \right)^2}.
  \end{align}
%%%%%%%%%%%%%%%%%%%%%%%%%%%%%%%%%%%%%%%%%%%%%%%%%%%%%%%%%%%%%%%%%%%%%%%%%%%%%%%%%%%%%%%%%%%%

Combining equations \eqref{dergam1} and \eqref{dergam2}, and using the fact that $\det \gamma = 1$, we compute 
 \begin{align*}
 d(\gamma z_1) \wedge d(\gamma z_2) = \frac{ \left( \left(a_{12}b_2 - a_{22}b_1 \right) c_1  + \left( a_{21} b_1 - a_{11}b_2 \right)c_2 +\left(a_{11}a_{22} - a_{12} a_{21} \right) D\right)}{\left(c_1z_1 + c_2 z_2 +D \right)^3}  dz_1 \wedge dz_2 \\[0.12cm]
 =  \bigg( \frac{ \det \gamma}{\left(c_1z_1 + c_2 z_2 +D \right)^3} \bigg)  dz_1 \wedge dz_2  =  \frac{ dz_1 \wedge dz_2}{\left(CZ+D \right)^3} .
 \end{align*}
%%%%%%%%%%%%%%%%%%%%%%%%%%%%%%%%%%%%%%%%%%%%%%%%%%%%%%%%%%%%%%%%%%%%%%%%%%%%%%%%%%%%%%%%%%%%

From the above discussion, if follows that at $z=(z_1,z_2)^{t}\in X_{\G}$, any $\omega \in H^{0}(X_{\G}, \Lambda)$ is of the form
\begin{align*}
\omega(z)=f(z)dz_{1}\wedge dz_{2},
\end{align*}
%%%%%%%%%%%%%%%%%%%%%%%%%%%%%%%%%%%%%%%%%%%%%%%%%%%%%%%%%%%%%%%%%%%%%%%%%%%%%%%%%%%%%%%%%%%%
where $f \in \cals_3(\Gamma)$ is a cusp form of weight-$3$, with respect to $\G$. 
%%%%%%%%%%%%%%%%%%%%%%%%%%%%%%%%%%%%%%%%%%%%%%%%%%%%%%%%%%%%%%%%%%%%%%%%%%%%%%%%%%%%%%%%%%%%

\vspace{0.1cm}
So, for any $k \geq 1$, at $z=(z_1,z_2)^{t}\in X_{\G}$, any $\omega \in H^{0}(X_{\G}, \Lambda^{\otimes k})$ is of the form
\begin{align*}
\omega(z)=f(z)\big(dz_{1}\wedge dz_{2}\big)^{\otimes k},
\end{align*}
%%%%%%%%%%%%%%%%%%%%%%%%%%%%%%%%%%%%%%%%%%%%%%%%%%%%%%%%%%%%%%%%%%%%%%%%%%%%%%%%%%%%%%%%%%%%
where $f \in \cals_{3k}(\Gamma)$ is a cusp form of weight-$3k$, with respect to $\G$. 
%%%%%%%%%%%%%%%%%%%%%%%%%%%%%%%%%%%%%%%%%%%%%%%%%%%%%%%%%%%%%%%%%%%%%%%%%%%%%%%%%%%%%%%%%%%%

\vspace{0.1cm}
Hence, for any $k\geq 1$, we have the isometry
\begin{align}\label{iso}
H^{0}\big(X, \Lambda^{\otimes k}\big)\cong \cals_{3k}(\G)
\end{align}
%%%%%%%%%%%%%%%%%%%%%%%%%%%%%%%%%%%%%%%%%%%%%%%%%%%%%%%%%%%%%%%%%%%%%%%%%%%%%%%%%%%%%%%%%%%%

For any $k\geq 1$, let $|\cdot|_{\mathrm{hyp}}$ and $\langle\cdot,\cdot\rangle_{\mathrm{hyp}}$ denote the point-wise and $L^2$ inner-product on
 $H^{0}(X_{\G},\Lambda^{\otimes k})$, respectively, which are induced by the hyperbolic metric $\hyp$. From isomorphism \eqref{iso}, at any $z\in X_{\G}$, and $w(z)=f(z)(dz_{1}\wedge dz_{2})^{\otimes k},\,\eta(z)=g(z)(dz_{1}\wedge dz_{2})^{\otimes k}$, we have
\begin{align}
&\big|\omega(z)\big|_{\mathrm{hyp}}^{2}=|f(z)|^{2}_{\mathrm{pet}}=\big(1-|z|^{2}\big)^{k}|f(z)|^{2};\notag\\[0.12cm]
&\langle \omega, \eta\rangle_{\mathrm{hyp}}= \big\langle f, g\big\rangle_{\mathrm{pet}}=\int_{\calf_{\G}}\big(1-|z|^{2}\big)^{k}f(z)\overline{g(z)}\hypvol(z).
\end{align}
%%%%%%%%%%%%%%%%%%%%%%%%%%%%%%%%%%%%%%%%%%%%%%%%%%%%%%%%%%%%%%%%%%%%%%%%%%%%%%%%%%%%%%%%%%%%
%%%%%%%%%%%%%%%%%%%%%%%%%%%%%%%%%%%%%%%%%%%%%%%%%%%%%%%%%%%%%%%%%%%%%%%%%%%%%%%%%%%%%%%%%%%%

\vspace{0.2cm}
\subsection{Bergman kernels and Bergman metric}\label{subsec-2.4}
%%%%%%%%%%%%%%%%%%%%%%%%%%%%%%%%%%%%%%%%%%%%%%%%%%%%%%%%%%%%%%%%%%%%%%%%%%%%%%%%%%%%%%%%%%%%
For any $k\geq 1$, let $\big\lbrace \omega_{1},\ldots ,\omega_{d_{k}} \rbrace$ denote an orthonormal basis of $H^{0}(X_{\G},\Lambda^{\otimes k} )$, with respect to the $L^2$ inner-product on $H^{0}(X_{\G},\Lambda^{\otimes k} )$, where $d_{k}$ is the dimension of $H^{0}(X_{\G},\Lambda^{\otimes k} )$. For any $z,w\in \Bt$, the Bergman kernel associated to the line bundle $\Lambda^{\otimes k}$ is given by the following  formula 
\begin{align}\label{bkdefn}
\blk(z,w):= \sum_{j=1}^{d_k}\omega_{j}(z)\overline{\omega_{j}(w)}.
\end{align}
%%%%%%%%%%%%%%%%%%%%%%%%%%%%%%%%%%%%%%%%%%%%%%%%%%%%%%%%%%%%%%%%%%%%%%%%%%%%%%%%%%%%%%%%%%%% 
From Riesz representation theorem, it follows that the definition of the Bergman kernel is independent of the choice of orthonormal basis of $H^{0}(X_{\G},\Lambda^{\otimes k} )$. 
%%%%%%%%%%%%%%%%%%%%%%%%%%%%%%%%%%%%%%%%%%%%%%%%%%%%%%%%%%%%%%%%%%%%%%%%%%%%%%%%%%%%%%%%%%%% 

\vspace{0.1cm}
Let $\lbrace f_{1},\ldots, f_{l_{k}} \rbrace$ denote an orthonormal basis of $\cals_{k}(\Gamma)$ with respect to the Petersson inner-product.  For any $z,w\in \Bt$, the Bergman kernel associated to the complex vector space $\cals_{k}(\G)$, is given by the following  formula 
\begin{align}\label{bkdefn}
\bsk(z,w):= \sum_{j=1}^{l_k}f_{j}(z)\overline{f_{j}(w)}.
\end{align}
%%%%%%%%%%%%%%%%%%%%%%%%%%%%%%%%%%%%%%%%%%%%%%%%%%%%%%%%%%%%%%%%%%%%%%%%%%%%%%%%%%%%%%%%%%%% 
From Riesz representation theorem, it follows that the definition of the Bergman kernel is independent of the choice of orthonormal basis of $\cals_{k}(\G)$. 
%%%%%%%%%%%%%%%%%%%%%%%%%%%%%%%%%%%%%%%%%%%%%%%%%%%%%%%%%%%%%%%%%%%%%%%%%%%%%%%%%%%%%%%%%%%% 

\vspace{0.2cm}
For any $k\geq 5$, and $z,w\in\Bt$, we have the following infinite series expression for the Bergman kernel $\bsk$
\begin{align}\label{bkseries}
\bsk(z,w)=\sum_{\gamma\in\G}\frac{\ck}{\big(-\langle \tilde{z},\tilde{\gamma w}\rangle_{\mathrm{hyp}}\big)^{k}( Cw+D)^{k}}, 
\end{align}
%%%%%%%%%%%%%%%%%%%%%%%%%%%%%%%%%%%%%%%%%%%%%%%%%%%%%%%%%%%%%%%%%%%%%%%%%%%%%%%%%%%%%%%%%%%% 
where $\tilde{z},\tilde{\gamma w}$ denote the lifts of $z,\gamma w$ to $\cala^3$, respectively, and the inner-product $\langle \cdot,\cdot\rangle_{\mathrm{hyp}}$ is as defined in equation  \eqref{hyp-ip}, and
\begin{align*}
\gamma=\left(\begin{smallmatrix} A&B\\C &D\end{smallmatrix}\right) \in \Gamma,\,\,\mathrm{and}\,\, A\in M_{2\times 2}(\C),\,\, B\in M_{2\times 1}(\C),\,\, C\in M_{1\times 2}(\C),\,\,\mathrm{and}\,\, D\in\C,
\end{align*}
%%%%%%%%%%%%%%%%%%%%%%%%%%%%%%%%%%%%%%%%%%%%%%%%%%%%%%%%%%%%%%%%%%%%%%%%%%%%%%%%%%%%%%%%%%%%
and the constant $C_{X_{\G},k}$ satisfies the following estimate
\begin{align}\label{defncgamma}
\ck=O_{X_{\G}}\big( k^2\big),
\end{align}
%%%%%%%%%%%%%%%%%%%%%%%%%%%%%%%%%%%%%%%%%%%%%%%%%%%%%%%%%%%%%%%%%%%%%%%%%%%%%%%%%%%%%%%%%%%%
where the implied constant depends only on the Picard surface $X_{\G}$.
%%%%%%%%%%%%%%%%%%%%%%%%%%%%%%%%%%%%%%%%%%%%%%%%%%%%%%%%%%%%%%%%%%%%%%%%%%%%%%%%%%%%%%%%%%%%

\vspace{0.1cm}
For any $k\geq 1$, and $z,w\in \Bt$, from isomorphism \eqref{iso}, and combining equations \eqref{bkdefn} and \eqref{bkseries}, we have the following relation
\begin{align}\label{bkreln}
\big|\blk(z,w)\big|_{\mathrm{hyp}}=\big|\bk(z,w)\big|_{\mathrm{pet}}= \big( 1- |z|^2\big)^{3k} \big| \bk(z,w)\big|.
\end{align}
%%%%%%%%%%%%%%%%%%%%%%%%%%%%%%%%%%%%%%%%%%%%%%%%%%%%%%%%%%%%%%%%%%%%%%%%%%%%%%%%%%%%%%%%%%%%

With notation as above, for any $k\geq 3$, and $z,w\in X_{\G}$, from equations \eqref{bkreln} and \eqref{bkseries}, we have
\begin{align}\label{bkineq1}
\big|\blk(z,w)\big|_{\mathrm{hyp}}= \big( 1- |z|^2\big)^{3k} \big| \bk(z,w)\big|\leq\sum_{\gamma\in\G}\frac{\cthreek\big( 1- |z|^2\big)^{3k\slash 2}\big( 1- |w|^2\big)^{3k\slash 2}}{\big|\big\langle \tilde{z},\tilde{\gamma w}\big\rangle_{\mathrm{hyp}}\big|^{3k}\big|Cw+D\big|^{3k}} .
\end{align}
%%%%%%%%%%%%%%%%%%%%%%%%%%%%%%%%%%%%%%%%%%%%%%%%%%%%%%%%%%%%%%%%%%%%%%%%%%%%%%%%%%%%%%%%%%%%

Observe that, for any 
\begin{align*}
\gamma:=\left(\begin{smallmatrix} a_{11}&a_{12}&b_1\\a_{21}&a_{22}&b_2\\c_{1}&c_{2}&D \end{smallmatrix}\right)\in \mathrm{SU}\big((2,1),\C \big),
\end{align*}
%%%%%%%%%%%%%%%%%%%%%%%%%%%%%%%%%%%%%%%%%%%%%%%%%%%%%%%%%%%%%%%%%%%%%%%%%%%%%%%%%%%%%%%%%%%%
and $w\in \Bt$, we have
\begin{align}\label{gwnorm}
&\big( 1-|\gamma w|^{2}\big)=-\big\langle \tilde{\gamma w},\tilde{\gamma w}\rangle_{\mathrm{hyp}}=-\big\langle \big((\gamma w)_{1},(\gamma w)_{2},1\big)^{t},\big((\gamma w)_{1},
(\gamma w)_{2},1\big)^{t}\big\rangle_{\mathrm{hyp}}=\notag\\[0.1cm]&-\frac{(\gamma (\tilde{w}))^{\ast}H\gamma (\tilde{w})}{\big|Cw+D\big|^{2}}=-\frac{\tilde{w}^{\ast}\gamma^{\ast}H
\gamma \tilde{w}}{\big|Cw+D\big|^{2}}=-\frac{\langle \tilde{w},\tilde{w}\rangle_{\mathrm{hyp}}}{\big|Cw+D\big|^{2}}=\frac{1-|w|^{2}}{\big|Cw+D\big|^{2}},
\end{align}
%%%%%%%%%%%%%%%%%%%%%%%%%%%%%%%%%%%%%%%%%%%%%%%%%%%%%%%%%%%%%%%%%%%%%%%%%%%%%%%%%%%%%%%%%%%%
where 
\begin{align*}
\tilde{\gamma w}=\big((\gamma w)_1,(\gamma w)_2,1\big)^{t}=\bigg(\frac{a_{11}w_{1}+a_{12}w_{2}+b_1}{c_1z_1+c_2 z_2+D},\frac{a_{21}w_{1}+a_{22}w_{2}+b_2}{c_1z_1+c_2 z_2+D},1\bigg)^{t}
\end{align*}
%%%%%%%%%%%%%%%%%%%%%%%%%%%%%%%%%%%%%%%%%%%%%%%%%%%%%%%%%%%%%%%%%%%%%%%%%%%%%%%%%%%%%%%%%%%%

Combining estimate \eqref{bkineq1} with equations \eqref{gwnorm} and \eqref{defn:hyp}, we arrive at the following estimate 
\begin{align}\label{bkineq}
&\big|\blk(z,w)\big|_{\mathrm{hyp}}= \big( 1- |z|^2\big)^{3k\slash 2} \big( 1-|\gamma w|\big)^{3k\slash 2}\big| \bk(z,w)\big|\leq\notag\\[0.1cm]&\sum_{\gamma\in\G} \frac{\cthreek\big( 1-|z|\big)^{3k\slash 2} \big( 1-|\gamma w|\big)^{3k\slash 2}}{\big|\big\langle \tilde{z},\tilde{\gamma w}\big\rangle_{\mathrm{hyp}}\big|^{3k}} =\sum_{\gamma\in\G}\frac{\cthreek \big|\langle\tilde{z},\tilde{z}\rangle_{\mathrm{hyp}}\big|^{3k\slash 2}\big|\langle\tilde{\gamma w},
\tilde{\gamma w}\rangle_{\mathrm{hyp}}\big|^{3k\slash 2}} {\big|\big\langle \tilde{z},\tilde{\gamma w}\big\rangle_{\mathrm{hyp}}\big|^{3k}}=\notag\\[0.1cm] 
&\hspace{6.8cm}\sum_{\gamma\in\G}\frac{\cthreek}{\cosh^{3k}\big(\dhyp(z,\gamma w)\slash 2\big)}.
\end{align}
%%%%%%%%%%%%%%%%%%%%%%%%%%%%%%%%%%%%%%%%%%%%%%%%%%%%%%%%%%%%%%%%%%%%%%%%%%%%%%%%%%%%%%%%%%%%

For any $k \geq 1$ and $z \in X_{\G}$, the Bergman metric associated to  the line bundle $\Lambda^{\otimes k}$ is given by the following formula
\begin{align} \label{bkmetric1} 
\bermet(z):= -\frac{i}{2\pi}\partial\overline{\partial}\log \big|\blk (z,z)\big|_{\mathrm{hyp}};
\end{align}
%%%%%%%%%%%%%%%%%%%%%%%%%%%%%%%%%%%%%%%%%%%%%%%%%%%%%%%%%%%%%%%%%%%%%%%%%%%%%%%%%%%%%%%%%%%%
and the associated volume form is given by the following formula
\begin{align}
\bervol(z):= \frac{1}{2!}{\bigwedge}^2 \big( \bermet(z) \big),
\end{align}
%%%%%%%%%%%%%%%%%%%%%%%%%%%%%%%%%%%%%%%%%%%%%%%%%%%%%%%%%%%%%%%%%%%%%%%%%%%%%%%%%%%%%%%%%%%%
where ${\bigwedge}^2$ denotes the $2$-wedge product.
%%%%%%%%%%%%%%%%%%%%%%%%%%%%%%%%%%%%%%%%%%%%%%%%%%%%%%%%%%%%%%%%%%%%%%%%%%%%%%%%%%%%%%%%%%%%

\vspace{0.1cm}
Combining equations \eqref{bkmetric1}, \eqref{bkreln}, and \eqref{defnhypmetric}, we arrive at the following equation
\begin{align} \label{bkmetric2} 
\bermet(z) = -\frac{i}{2\pi} \partial \overline{\partial}\log \big( 1-|z|^2 \big)^{3k}  -\frac{i}{2\pi}\partial \overline{\partial}\log \big|\bk (z,z)\big|  = \notag \\ 
  \frac{3k}{\pi}\hypbn -\frac{i}{2\pi}  \partial \overline{\partial}\log \big| \bk (z,z) \big| .
\end{align}
%%%%%%%%%%%%%%%%%%%%%%%%%%%%%%%%%%%%%%%%%%%%%%%%%%%%%%%%%%%%%%%%%%%%%%%%%%%%%%%%%%%%%%%%%%%%
%%%%%%%%%%%%%%%%%%%%%%%%%%%%%%%%%%%%%%%%%%%%%%%%%%%%%%%%%%%%%%%%%%%%%%%%%%%%%%%%%%%%%%%%%%%%

\vspace{0.1cm}
\subsection{Counting functions}\label{subsec-2.5}
%%%%%%%%%%%%%%%%%%%%%%%%%%%%%%%%%%%%%%%%%%%%%%%%%%%%%%%%%%%%%%%%%%%%%%%%%%%%%%%%%%%%%%%%%%%%
We now describe the counting functions associated to estimate the number of elements of the cocompact subgroup $\G\subset \mathrm{SU}((2,1),\C)$. For any $z,w\in X_{\G}$, put 
\begin{align}\label{defncountingfn}
\cals_{\Gamma}(z,w;\delta):=&\,\big\lbrace \gamma\in \Gamma| \,\dhyp(z,\gamma w)\leq \delta\big\rbrace, \notag\\
N_{\Gamma}(z,w;\delta):=&\,\mathrm{card}\,\cals_{\Gamma}(z,w;\delta).
\end{align}
%%%%%%%%%%%%%%%%%%%%%%%%%%%%%%%%%%%%%%%%%%%%%%%%%%%%%%%%%%%%%%%%%%%%%%%%%%%%%%%%%%%%%%%%%%%%

\vspace{0.1cm}
Let $\mathrm{vol}(B(z,r))$ denote the volume of a geodesic ball of radius $r$, centred at any $z\in\Bt$. From arguments from elementary hyperbolic geometry, we have the following formula
\begin{align}\label{hypball}
\mathrm{vol}\big(B(z,r)\big)=2\pi \sinh^{4}(r\slash 2).
\end{align}
%%%%%%%%%%%%%%%%%%%%%%%%%%%%%%%%%%%%%%%%%%%%%%%%%%%%%%%%%%%%%%%%%%%%%%%%%%%%%%%%%%%%%%%%%%%%

\vspace{0.1cm}
We now describe estimates of $\cals_{\Gamma}(z,w;\delta)$ from Lemma 4.1 in \cite{ab}, which is generalization of Lemma $4$ in \cite{jl}. For any $\delta >0$, we have the following estimate 
\begin{align}\label{countingfnestimate1}
N_{\Gamma}(z,w;\delta)\leq \frac{\sinh^{4}\big((2\delta+\rx)\slash4\big)}{\sinh^{4}(\rx\slash 4)}.
\end{align} 
%%%%%%%%%%%%%%%%%%%%%%%%%%%%%%%%%%%%%%%%%%%%%%%%%%%%%%%%%%%%%%%%%%%%%%%%%%%%%%%%%%%%%%%%%%%%

Let $f$ be any positive, smooth, monotonically decreasing function defined on $\R_{>0}$. Then, for any $\delta> \rx\slash 2$, assuming that all the involved integrals exist, we have the following inequality
\begin{align}\label{countingfnestimate2}
\int_{0}^{\infty}f(\rho)dN_{\Gamma}(z,w;\rho)\leq \int_{0}^{\delta}f(\rho)dN_{\Gamma}(z,w;\rho)+f(\delta)\frac{\sinh^{4}\big((2\delta+\rx)\slash 4\big)}{\sinh^{4}(\rx\slash 4)}+\notag
\\[0.13cm]\frac{1}{\sinh^{4}(\rx\slash 4)}\int_{\delta}^{\infty}f(\rho)\sinh^{3}\big((2\rho+\rx)\slash 4\big)\cosh((2\rho+\rx)\slash4)d\rho.
\end{align}
%%%%%%%%%%%%%%%%%%%%%%%%%%%%%%%%%%%%%%%%%%%%%%%%%%%%%%%%%%%%%%%%%%%%%%%%%%%%%%%%%%%%%%%%%%%%

The estimate derived in \cite{ab} has a few extra factors of $2\pi$, which are taken into account in estimate \eqref{countingfnestimate2}.
%%%%%%%%%%%%%%%%%%%%%%%%%%%%%%%%%%%%%%%%%%%%%%%%%%%%%%%%%%%%%%%%%%%%%%%%%%%%%%%%%%%%%%%%%%%%
%%%%%%%%%%%%%%%%%%%%%%%%%%%%%%%%%%%%%%%%%%%%%%%%%%%%%%%%%%%%%%%%%%%%%%%%%%%%%%%%%%%%%%%%%%%%

\vspace{0.2cm}
\section{Proofs of Main Theorem \ref{mainthm1} and Main Theorem \ref{mainthm2} }\label{sec3}
%%%%%%%%%%%%%%%%%%%%%%%%%%%%%%%%%%%%%%%%%%%%%%%%%%%%%%%%%%%%%%%%%%%%%%%%%%%%%%%%%%%%%%%%%%%%
In this section, building on the notions and notation introduced in Section \ref{sec2}, we prove Main Theorem \ref{mainthm1} and Main Theorem \ref{mainthm2} as Theorem \ref{thm4} and Theorem \ref{thm9}, respectively. 
%%%%%%%%%%%%%%%%%%%%%%%%%%%%%%%%%%%%%%%%%%%%%%%%%%%%%%%%%%%%%%%%%%%%%%%%%%%%%%%%%%%%%%%%%%%%
%%%%%%%%%%%%%%%%%%%%%%%%%%%%%%%%%%%%%%%%%%%%%%%%%%%%%%%%%%%%%%%%%%%%%%%%%%%%%%%%%%%%%%%%%%%%

\vspace{0.2cm}
\subsection{Proof of Main Theorem \ref{mainthm1}}\label{subsec-3.1}
%%%%%%%%%%%%%%%%%%%%%%%%%%%%%%%%%%%%%%%%%%%%%%%%%%%%%%%%%%%%%%%%%%%%%%%%%%%%%%%%%%%%%%%%%%%%
\begin{prop}\label{prop1}
With notation as above, for any $k\geq 3$, and $z,w\in X_{\G}$ with $\dhyp(z,w)\geq3\rx\slash 4$, we have the following estimate
\begin{align*}
\big|\blk(z,w)\big|_{\mathrm{hyp}}\leq 
\frac{\cthreek}{\cosh^{3k}\big(\dhyp(z,w\slash 2\big)}
+\frac{16\coth\big(\rx\slash 4\big)\cthreek}{\cosh^{3k}\big(\dhyp(z,w)\slash 2\big)}+\notag\\[0.12cm]
\frac{32\cthreek}{(3k-6)\cosh^{3k-6}\big(\dhyp(z,w)\slash 2\big)},
\end{align*}
%%%%%%%%%%%%%%%%%%%%%%%%%%%%%%%%%%%%%%%%%%%%%%%%%%%%%%%%%%%%%%%%%%%%%%%%%%%%%%%%%%%%%%%%%%%%
where $\rx$ is the injectivity radius of $X$, which is as defined in equation \eqref{ircom}, and $\cthreek$ is as described in equation \eqref{defncgamma}. 
%%%%%%%%%%%%%%%%%%%%%%%%%%%%%%%%%%%%%%%%%%%%%%%%%%%%%%%%%%%%%%%%%%%%%%%%%%%%%%%%%%%%%%%%%%%%
\end{prop}
%%%%%%%%%%%%%%%%%%%%%%%%%%%%%%%%%%%%%%%%%%%%%%%%%%%%%%%%%%%%%%%%%%%%%%%%%%%%%%%%%%%%%%%%%%%%
\begin{proof}
For any $k\geq 3$, and any $z,w\in X_{\G}$ with $\dhyp(z,w)\geq3\rx\slash 4$, from inequality \eqref{bkineq}, we have
\begin{align}\label{prop1eqn1}
\big|\blk(z,w)\big|_{\mathrm{hyp}}\leq\sum_{\gamma\in\Gamma}\frac{\cthreek}{\cosh^{3k}\big(\dhyp(z,\gamma w)\slash 2\big)}.
\end{align}
%%%%%%%%%%%%%%%%%%%%%%%%%%%%%%%%%%%%%%%%%%%%%%%%%%%%%%%%%%%%%%%%%%%%%%%%%%%%%%%%%%%%%%%%%%%%
We fix a $z\in \Bt$, and work with the following Dirichilet fundamental domain
\begin{align}\label{prop1eqn2}
\calf_{\G,z}:=\big\lbrace w\in \Bt\big|\dhyp(z,w)<\dhyp(z,\gamma w), \,\,\mathrm{for\,\,all}\,\,\gamma\in\G\backslash\lbrace\mathrm{Id}\rbrace\big\rbrace,
\end{align}
%%%%%%%%%%%%%%%%%%%%%%%%%%%%%%%%%%%%%%%%%%%%%%%%%%%%%%%%%%%%%%%%%%%%%%%%%%%%%%%%%%%%%%%%%%%%
and choose our $w\in \calf_{\G,z}$.
%%%%%%%%%%%%%%%%%%%%%%%%%%%%%%%%%%%%%%%%%%%%%%%%%%%%%%%%%%%%%%%%%%%%%%%%%%%%%%%%%%%%%%%%%%%%

\vspace{0.1cm}
Now, as $1\slash \cosh^{3k}(\rho\slash 2)$ is a positive, smooth, monotonically decreasing function on $\R_{>0}$,  substituting $\delta=\dhyp(z,w)$ in estimate  \eqref{countingfnestimate2}, we compute
\begin{align}\label{prop1eqn3}
\sum_{\gamma\in\Gamma}\frac{\cthreek}{\cosh^{3k}\big(\dhyp(z,\gamma w)\slash 2\big)}=\int_{0}^{\infty}\frac{\cthreek dN_{\Gamma}(z,\gamma w;\rho)}{\cosh^{3k}(\rho\slash 2)}\leq \notag\\[0.12cm]
\int_{0}^{\dhyp(z,w)}\frac{\cthreek dN_{\Gamma}(z, w;\rho)}{\cosh^{3k}(\rho\slash 2)}+\frac{\sinh^{4}\big((2\dhyp(z,w)+\rx)\slash 4\big)\cthreek}{\sinh^{4}(\rx\slash 4)\cosh^{3k}\big(\dhyp(z,w)\slash 2\big)}+\notag\\[0.13cm]\frac{\cthreek}{\sinh^{4}(\rx\slash 4)}\int_{\dhyp(z,w)}^{\infty}\frac{\sinh^{3}\big((2\rho+\rx)\slash 4\big)\cosh\big((2\rho+\rx)\slash4\big)}{\cosh^{3k}(\rho\slash 2)}d\rho.
\end{align}
%%%%%%%%%%%%%%%%%%%%%%%%%%%%%%%%%%%%%%%%%%%%%%%%%%%%%%%%%%%%%%%%%%%%%%%%%%%%%%%%%%%%%%%%%%%%

We now estimate each of the three terms on the right hand-side of the above inequality. From the choice of fundamental domain, which is as described in \eqref{prop1eqn2}, we have the following estimate for the first term on the right hand-side of inequality \eqref{prop1eqn3}
\begin{align}\label{prop1eqn4}
\int_{0}^{\dhyp(z,w)}\frac{\cthreek dN_{\Gamma}(z,\gamma w;\rho)}{\cosh^{3k}(\rho\slash 2)}=\frac{\cthreek}{\cosh^{3k}\big(\dhyp(z,w)\slash 2\big)}.
\end{align}
%%%%%%%%%%%%%%%%%%%%%%%%%%%%%%%%%%%%%%%%%%%%%%%%%%%%%%%%%%%%%%%%%%%%%%%%%%%%%%%%%%%%%%%%%%%%

Observe that, we have
\begin{align*}
\sinh\big((2\rho+\rx)\slash 4\big)\leq 2\cosh(\rx\slash 4)\cosh(\rho\slash 2).
\end{align*}
%%%%%%%%%%%%%%%%%%%%%%%%%%%%%%%%%%%%%%%%%%%%%%%%%%%%%%%%%%%%%%%%%%%%%%%%%%%%%%%%%%%%%%%%%%%%

Using the above inequality, we derive the following estimate for the second term on  the right hand-side of inequality  \eqref{prop1eqn3}
\begin{align}\label{prop1eqn5}
\frac{\cthreek\sinh^{4}\big((2\dhyp(z,w)+\rx)\slash 4\big)}{\sinh^{4}\big(\rx\slash 4\big)\cosh^{3k}\big(\dhyp(z,w)\slash 2\big)}\leq \frac{16\coth^{4}\big(\rx\slash 4\big)\cthreek}{\cosh^{3k-4}\big(\dhyp(z,w)\slash 2\big)}.
\end{align}
%%%%%%%%%%%%%%%%%%%%%%%%%%%%%%%%%%%%%%%%%%%%%%%%%%%%%%%%%%%%%%%%%%%%%%%%%%%%%%%%%%%%%%%%%%%%

Observe that for any $\rho\geq\frac{3\rx}{4}$, we have
\begin{align*}
&\sinh\big((2\rho+\rx)\slash 4\big)\leq \sin(\rho)=2\sinh(\rho\slash 2) \cosh(\rho\slash 2);\\[0.12cm]
&\sinh^{2}\big((2\rho+\rx)\slash 4\big)\leq 4\cosh^{2}(\rx\slash 4)\cosh^{2}(\rho\slash 2)\leq 4\cosh^{4}\big(\rx\slash 4\big)\cosh^{2}(\rho\slash 2);\\[0.12cm]&\cosh\big((2\rho+\rx)\slash4\big)\leq \cosh(\rho)\leq 2\cosh^{2}(\rho\slash 2).
\end{align*}
%%%%%%%%%%%%%%%%%%%%%%%%%%%%%%%%%%%%%%%%%%%%%%%%%%%%%%%%%%%%%%%%%%%%%%%%%%%%%%%%%%%%%%%%%%%%

Using the above two inequalities, for any $k\geq 3$, we compute
\begin{align*}
\int_{\dhyp(z,w)}^{\infty}\frac{\sinh^{3}\big((2\rho+\rx)\slash 4\big)\cosh\big((2\rho+\rx)\slash4\big)}{\cosh^{3k}(\rho\slash 2)}d\rho  \leq \notag\\[0.12cm]
32\cosh^{4}(\rx\slash 4)\int_{\dhyp(z,w)}^{\infty}\frac{d\big(\cosh(\rho\slash 2)\big)}{\cosh^{3k-5}(\rho\slash 2)}=\frac{32\cosh^{4}(\rx\slash 4)}{(3k-6)\cosh^{3k-6}(\rho\slash 2)}.
\end{align*}
%%%%%%%%%%%%%%%%%%%%%%%%%%%%%%%%%%%%%%%%%%%%%%%%%%%%%%%%%%%%%%%%%%%%%%%%%%%%%%%%%%%%%%%%%%%%

From the above computation, we arrive at the following estimate for the third term on  the right hand-side of inequality  \eqref{prop1eqn3}
\begin{align}\label{prop1eqn6}
\frac{\cthreek}{\sinh^{4}\big(\rx\slash 4\big)}\int_{\dhyp(z,w)}^{\infty}\frac{\sinh^{3}\big((2\rho+\rx)\slash 4\big)\cosh\big((2\rho+\rx)\slash4\big)}{\cosh^{3k}\big(\rho\slash 2\big)}d\rho \leq 
\notag\\[0.12cm]\frac{32 \coth^{4}\big(\rx\slash 4\big)\cthreek}{(3k-6)\cosh^{3k-6}\big(\dhyp(z,w)\slash 2\big)}.
\end{align}
%%%%%%%%%%%%%%%%%%%%%%%%%%%%%%%%%%%%%%%%%%%%%%%%%%%%%%%%%%%%%%%%%%%%%%%%%%%%%%%%%%%%%%%%%%%%

Combining estimates \eqref{prop1eqn1}, \eqref{prop1eqn3}, \eqref{prop1eqn4}, \eqref{prop1eqn5}, and \eqref{prop1eqn6}, completes the proof of the lemma.
\end{proof}
%%%%%%%%%%%%%%%%%%%%%%%%%%%%%%%%%%%%%%%%%%%%%%%%%%%%%%%%%%%%%%%%%%%%%%%%%%%%%%%%%%%%%%%%%%%%

\vspace{0.2cm}
\begin{prop}\label{prop2}
With notation as above, for any $k\geq 3$, and $z,w\in X_{\G}$ with $\dhyp(z,w)<3\rx\slash 4$, we have the following estimate
\begin{align}\label{prop2:eqn}
\big|\blk(z,w)\big|_{\mathrm{hyp}}\leq 
\frac{\cthreek}{\cosh^{k}\big(\dhyp(z,w\slash 2\big)}+\frac{2\cthreek\sinh^{4}(5\rx\slash 8)}{\sinh^{4}(\rx\slash 4)\cosh^{3k}\big(\dhyp(z,w)\slash 2\big)}+
\notag\\[0.12cm]\frac{ 32\coth(\rx\slash 4)\cthreek}{\big(3k-8\big)\sinh^{4}(\rx\slash 4)\cosh^{3k-8}(3\rx\slash 4)},
\end{align}
%%%%%%%%%%%%%%%%%%%%%%%%%%%%%%%%%%%%%%%%%%%%%%%%%%%%%%%%%%%%%%%%%%%%%%%%%%%%%%%%%%%%%%%%%%%%
where $\rx$ is the injectivity radius of $X$, which is as defined in equation \eqref{ircom}, and $\cthreek$ is as described in equation \eqref{defncgamma}. 
%%%%%%%%%%%%%%%%%%%%%%%%%%%%%%%%%%%%%%%%%%%%%%%%%%%%%%%%%%%%%%%%%%%%%%%%%%%%%%%%%%%%%%%%%%%%
\end{prop}
%%%%%%%%%%%%%%%%%%%%%%%%%%%%%%%%%%%%%%%%%%%%%%%%%%%%%%%%%%%%%%%%%%%%%%%%%%%%%%%%%%%%%%%%%%%%
\begin{proof}
For any $k\geq 3$, and any $z,w\in X_{\G}$ with $\dhyp(z,w)<3\rx\slash 4$, from arguments similar to the ones employed in Proposition \ref{prop1} (i.e., substituting $\delta=3\rx\slash 4$ in estimate \eqref{countingfnestimate2} and combining it with estimate \eqref{prop1eqn1}), we arrive at the following estimate
\begin{align}\label{prop2eqn1}
\big|\blk(z,w)\big|_{\mathrm{hyp}}\leq\int_{0}^{3\rx\slash 4}\frac{\cthreek dN_{\Gamma}(z,\gamma w;\rho)}{\cosh^{3k}(\rho\slash 2)}+
\frac{\sinh^{4}(5\rx\slash 8)\cthreek}{\sinh^{4}(\rx\slash 4)\cosh^{3k}\big(\dhyp(z,w)\slash 2\big)}+\notag\\[0.12cm]\frac{\cthreek}{\sinh^{4}(\rx\slash 4)}\int_{3\rx\slash 4}^{\infty}\frac{\sinh^{3}\big((2\rho+\rx)\slash 4\big)\cosh\big((2\rho+\rx)\slash4\big)}{\cosh^{3k}(\rho\slash 2)}d\rho.
\end{align}
%%%%%%%%%%%%%%%%%%%%%%%%%%%%%%%%%%%%%%%%%%%%%%%%%%%%%%%%%%%%%%%%%%%%%%%%%%%%%%%%%%%%%%%%%%%%

We now estimate each of the two integrals on the right hand-side of the above inequality. Observe that
\begin{align}\label{prop2eqn2}
\int_{0}^{3\rx\slash 4}\frac{ dN_{\Gamma}(z,\gamma w;\rho)}{\cosh^{3k}\big(\rho\slash 2\big)}=\int_{0}^{\dhyp(z,w)}\frac{ dN_{\Gamma}(z,\gamma w;\rho)}{\cosh^{3k}\big(\rho\slash 2\big)}+\int_{\dhyp(z,w)}^{3\rx\slash 4}\frac{dN_{\Gamma}(z,\gamma w;\rho)}{\cosh^{3k}\big(\rho\slash 2\big)}.
\end{align}
%%%%%%%%%%%%%%%%%%%%%%%%%%%%%%%%%%%%%%%%%%%%%%%%%%%%%%%%%%%%%%%%%%%%%%%%%%%%%%%%%%%%%%%%%%%%

Using the fact that $\dhyp(z,w)<3\rx\slash 4$, and substituting $\delta=3\rx\slash 4$ in estimate \eqref{countingfnestimate1}, we arrive at the following estimates
\begin{align}\label{prop2eqn3}
&\int_{0}^{\dhyp(z,w)}\frac{ dN_{\Gamma}(z,\gamma w;\rho)}{\cosh^{3k}\big(\rho\slash 2\big)}=\frac{1}{\cosh^{3k}\big(\dhyp(z,w)\slash 2\big)};\notag\\[0.12cm]
&\int_{\dhyp(z,w)}^{3\rx\slash 4}\frac{ dN_{\Gamma}(z,\gamma w;\rho)}{\cosh^{3k}\big(\rho\slash 2\big)}\leq \frac{\sinh^{4}\big(5\rx\slash 8\big)}{\sinh^{4}\big(\rx\slash 4\big)\cosh^{3k}\big(\dhyp(z,w)\slash 2\big)}.
\end{align}
%%%%%%%%%%%%%%%%%%%%%%%%%%%%%%%%%%%%%%%%%%%%%%%%%%%%%%%%%%%%%%%%%%%%%%%%%%%%%%%%%%%%%%%%%%%%

Combining estimates \eqref{prop2eqn2} and \eqref{prop2eqn3}, we arrive at the following estimate for the first integral on the right hand-side of inequality \eqref{prop2eqn1}
\begin{align}\label{prop2eqn4}
\int_{0}^{3\rx\slash 4}\frac{\cthreek dN_{\Gamma}(z,\gamma w;\rho)}{\cosh^{3k}\big(\rho\slash 2\big)}\leq\frac{\cthreek }{\cosh^{3k}\big(\rho\slash 2\big)}+\frac{\sinh^{4}\big(5\rx\slash 8\big)\cthreek}{\sinh^{4}\big(\rx\slash 4\big)\cosh^{3k}\big(\dhyp(z,w)\slash 2\big)}.
\end{align}
%%%%%%%%%%%%%%%%%%%%%%%%%%%%%%%%%%%%%%%%%%%%%%%%%%%%%%%%%%%%%%%%%%%%%%%%%%%%%%%%%%%%%%%%%%%%

Observe that for any $\rho\geq\frac{3\rx}{4} $, we have 
\begin{align*}
\sinh\big((2\rho+\rx)\slash 4\big)=\sinh\big(\rho\slash 2\big)\cosh\big(\rx\slash 4\big)+\cosh\big(\rho\slash 2\big)\sinh\big(\rx\slash 4\big)\leq \\[0.12cm]
\sinh\big(\rho\slash 2\big)\cosh\big(\rho\slash 3\big)+\cosh\big(\rho\slash 2\big)\sinh\big(\rho\slash 3\big)\leq 2\sinh\big(\rho\slash 2\big)\cosh\big(\rho\slash 2\big);
\end{align*}
%%%%%%%%%%%%%%%%%%%%%%%%%%%%%%%%%%%%%%%%%%%%%%%%%%%%%%%%%%%%%%%%%%%%%%%%%%%%%%%%%%%%%%%%%%%%
and 
\begin{align*}
\cosh\big((2\rho+\rx)\slash4\big)\leq 2\cosh^2\big(2\rho\slash2\big).
\end{align*}
%%%%%%%%%%%%%%%%%%%%%%%%%%%%%%%%%%%%%%%%%%%%%%%%%%%%%%%%%%%%%%%%%%%%%%%%%%%%%%%%%%%%%%%%%%%%

Using the above two inequalities, we derive the following estimate for the second  integral on the right hand-side of inequality \eqref{prop2eqn1}
\begin{align}\label{prop2eqn5}
\int_{3\rx\slash 4}^{\infty}\frac{\sinh^{3}\big((2\rho+\rx)\slash 4\big)\cosh\big((2\rho+\rx)\slash4\big)}{\cosh^{3k}\big(\rho\slash 2\big)}d\rho\leq 
16\int_{3\rx\slash 4}^{\infty}\frac{\sinh^3\big(\rho\slash 2\big)d\rho}{\cosh^{3k-5}\big(\rho\slash 2\big)}\leq\notag \\[0.12cm]
32\int_{3\rx\slash 4}^{\infty}\frac{d\big( \cosh\big(\rho\slash 2\big)\big)}{\cosh^{3k-7}\big(\rho\slash 2\big)}=\frac{32}{(3k-8)\cosh^{3k-8}\big(3\rx\slash 8\big)}.
\end{align}
%%%%%%%%%%%%%%%%%%%%%%%%%%%%%%%%%%%%%%%%%%%%%%%%%%%%%%%%%%%%%%%%%%%%%%%%%%%%%%%%%%%%%%%%%%%%
Combining the above estimate with estimates \eqref{prop2eqn1} and \eqref{prop2eqn4}, completes the proof of the proposition.
 \end{proof}
%%%%%%%%%%%%%%%%%%%%%%%%%%%%%%%%%%%%%%%%%%%%%%%%%%%%%%%%%%%%%%%%%%%%%%%%%%%%%%%%%%%%%%%%%%%%

\vspace{0.2cm}
\begin{cor}\label{cor3}
With notation as above, for any $k\geq 3$, and $z\in X_{\G}$, we have the following estimate
\begin{align}\label{ctildek}
&\big|\blk(z,z)\big|_{\mathrm{hyp}}\leq \sum_{\gamma\in\G}\frac{\cthreek}{\cosh^{3k}\big(\dhyp(z,\gamma z)\slash 2 \big)}\leq \ctildek,\notag\\[0.12cm]
&\ctildek:=
\cthreek+\frac{\cthreek\sinh^{4}(3\rx\slash 4)}{\sinh^{4}(\rx\slash 4)\cosh^{3k}(\rx\slash 2)}+\frac{32\cthreek}{(3k-8)\sinh^{4}(\rx\slash 4)\cosh^{3k-8}(\rx\slash 2)}.
\end{align}
\end{cor}
%%%%%%%%%%%%%%%%%%%%%%%%%%%%%%%%%%%%%%%%%%%%%%%%%%%%%%%%%%%%%%%%%%%%%%%%%%%%%%%%%%%%%%%%%%%%
\begin{proof}
For any $k\geq 3$, and $z\in X_{\G}$, from arguments as in the proofs of Propositions \ref{prop1} and \ref{prop2}, and substituting $\delta=\rx$ in estimate \eqref{countingfnestimate2}, we arrive at the following estimate
\begin{align}\label{cor3eqn1}
&\big|\blk(z,z)\big|_{\mathrm{hyp}}\leq \cthreek+\frac{\cthreek\sinh^{4}(3\rx\slash 4)}{\sinh^{4}(\rx\slash 4)\cosh^{3k}(\rx\slash 2)}+\notag\\[0.12cm]&\frac{\cthreek}{\sinh^{4}(\rx\slash 4)}\int_{\rx}^{\infty}\frac{\sinh^{3}\big((2\rho+\rx)\slash 4\big)\cosh\big((2\rho+\rx)\slash4\big)d\rho}{\cosh^{3k}(\rho\slash 2)}
\end{align}
%%%%%%%%%%%%%%%%%%%%%%%%%%%%%%%%%%%%%%%%%%%%%%%%%%%%%%%%%%%%%%%%%%%%%%%%%%%%%%%%%%%%%%%%%%%%

From arguments similar to the ones employed in proving estimate \eqref{prop2eqn5} in the proof of Proposition \ref{prop2}, we have the following estimate for the third term on the right hand-side of inequality \eqref{cor3eqn1}
\begin{align}\label{cor3eqn2}
\frac{\cthreek}{\sinh^{4}(\rx\slash 4)}\int_{\rx}^{\infty}\frac{\sinh^{3}\big((2\rho+\rx)\slash 4\big)\cosh\big((2\rho+\rx)\slash4\big)d\rho}{\cosh^{3k}(\rho\slash 2)}\leq\notag\\[0.12cm]
\frac{32\cthreek}{(3k-8)\sinh^{4}(\rx\slash 4)\cosh^{3k-8}(\rx\slash 2)}. 
\end{align}
%%%%%%%%%%%%%%%%%%%%%%%%%%%%%%%%%%%%%%%%%%%%%%%%%%%%%%%%%%%%%%%%%%%%%%%%%%%%%%%%%%%%%%%%%%%%

Combining estimates \eqref{cor3eqn1} and \eqref{cor3eqn2} completes the proof of the corollary. 
\end{proof}
%%%%%%%%%%%%%%%%%%%%%%%%%%%%%%%%%%%%%%%%%%%%%%%%%%%%%%%%%%%%%%%%%%%%%%%%%%%%%%%%%%%%%%%%%%%%

\vspace{0.2cm}
\begin{thm}\label{thm4}
With notation as above, for any $k\geq 3$, and $z,w\in X_{\G}$, we have the following estimate
\begin{align*}
\big|\blk(z,w)\big|_{\mathrm{hyp}} =O_{X_{\G}}\bigg(\frac{k^{2}}{\cosh^{3k-8}\big(\dhyp(z,w)\slash 2\big)}\bigg),
\end{align*}
%%%%%%%%%%%%%%%%%%%%%%%%%%%%%%%%%%%%%%%%%%%%%%%%%%%%%%%%%%%%%%%%%%%%%%%%%%%%%%%%%%%%%%%%%%%%
and the implied constant depends only on $X_{\G}$.
\end{thm}
%%%%%%%%%%%%%%%%%%%%%%%%%%%%%%%%%%%%%%%%%%%%%%%%%%%%%%%%%%%%%%%%%%%%%%%%%%%%%%%%%%%%%%%%%%%%
\begin{proof}
The proof of the theorem follows directly from combining Propositions \ref{prop1} and \ref{prop2} with estimate \eqref{defncgamma}.
\end{proof}
%%%%%%%%%%%%%%%%%%%%%%%%%%%%%%%%%%%%%%%%%%%%%%%%%%%%%%%%%%%%%%%%%%%%%%%%%%%%%%%%%%%%%%%%%%%%
%%%%%%%%%%%%%%%%%%%%%%%%%%%%%%%%%%%%%%%%%%%%%%%%%%%%%%%%%%%%%%%%%%%%%%%%%%%%%%%%%%%%%%%%%%%%

\vspace{0.2cm}
\subsection{Proof of Main Theorem \ref{mainthm2}}\label{subsec-3.2}
%%%%%%%%%%%%%%%%%%%%%%%%%%%%%%%%%%%%%%%%%%%%%%%%%%%%%%%%%%%%%%%%%%%%%%%%%%%%%%%%%%%%%%%%%%%
With notation as above, for any $z\in X_{\G}$, for brevity of notation, for the rest of the article, we denote $\bk(z,z)$ by $\bk(z)$. We will also work with a fixed fundamental domain $\calf_{\G}$ of $X_{\G}$, for the rest of the article. 
%%%%%%%%%%%%%%%%%%%%%%%%%%%%%%%%%%%%%%%%%%%%%%%%%%%%%%%%%%%%%%%%%%%%%%%%%%%%%%%%%%%%%%%%%%%

\vspace{0.2cm}
For any $z:=(z_1,z_2)^t\in  \Bt$, combining  equations \eqref{bkseries}  and \eqref{hyp-ip}, we observe that 
\begin{align}\label{bkseries2}
 \bk(z) = \sum_{\gamma\in\Gamma}\frac{\cthreek}{\big( \overline{\big( c_1 z_1+c_2 z_2+ D \big)} -  \overline{\big( a_{11}z_1+ a_{12}z_2 + b_1 \big) } z_1 - \overline{\big( a_{21}z_1+ a_{22}z_2 + b_2 \big)} z_2 \big)^{3k}},
\end{align}
%%%%%%%%%%%%%%%%%%%%%%%%%%%%%%%%%%%%%%%%%%%%%%%%%%%%%%%%%%%%%%%%%%%%%%%%%%%%%%%%%%
where any $\gamma\in\G$ is expressed as
\begin{align*}
 \gamma=\left(\begin{smallmatrix}A&B\\C&D \end{smallmatrix}\right),
\end{align*}
%%%%%%%%%%%%%%%%%%%%%%%%%%%%%%%%%%%%%%%%%%%%%%%%%%%%%%%%%%%%%%%%%%%%%%%%%%%%%%%%%%%%%%%%%%%
and
\begin{align*}
A= \left(\begin{smallmatrix} a_{11}&  a_{12} \\ a_{21} & a_{22}  \end{smallmatrix}\right)\in M_{2\times 2}(\C) ; \;\; B=(b_1,b_2)^{t}\in M_{2\times 1}(\C) ; \;\; C= (c_1,c_2) \in M_{1\times 2}(\C); \;\; \mathrm{and} \;\;D \in \C.
\end{align*}
%%%%%%%%%%%%%%%%%%%%%%%%%%%%%%%%%%%%%%%%%%%%%%%%%%%%%%%%%%%%%%%%%%%%%%%%%%%%%%%%%%%%%%%%%%%%

\vspace{0.1cm}
\begin{lem}\label{lem5}
With notation as above, for any $k\geq 3$ and $z:=(z_1,z_2)^{t} \in \calf_{\G}$, we have the following equality
\begin{align}\label{lem5:eqn}
 \frac{\bervol(z)}{\hypbnvol(z)} =\frac{\big(1-|z|^{2}\big)^{3}}{\pi^2} \big|T_1 + T_2 + T_3 + T_4 \big|
\end{align}
%%%%%%%%%%%%%%%%%%%%%%%%%%%%%%%%%%%%%%%%%%%%%%%%%%%%%%%%%%%%%%%%%%%%%%%%%%%%%%%%%%%%%%%%%%%%
where 
\begin{align}  
 T_1:=  \frac{ 9k^2} {\big(1-|z|^2\big)^3} , \;\;\; 
%%%%%%%%%%%%%%%%%%%%%%%%%%%%%%%%%%%%%%%%%%%%%%%%%%%%%%%%%%%%%%%%%%%%%%%%%%%%%%%%%%%%%%%%%%%%
 T_2 := \frac{3k}{\bk(z)}  \bigg( -\frac{ (1- |z_2|^2)} {(1-|z|^2)^2} \frac{\partial^2 \bk(z)} {\partial z_2 \partial\overline{ z_2}} - \frac{ (1- |z_1|^2)} {(1-|z|^2)^2} \frac{\partial^2 \bk(z)}{\partial z_1 \partial\overline{ z_1}} + \notag\\[0.12cm]  \frac{z_2 \overline{z_1}}{(1-|z|^2)^2} \frac{\partial^2 \bk(z)}{\partial z_2 \partial\overline{ z_1}}  +   \frac{z_1 \overline{z_2}}{(1-|z|^2)^2} \frac{\partial^2 \bk(z)}{\partial z_1 \partial\overline{ z_2}} \bigg), \label{lem5:eqn1}
 \end{align}    
%%%%%%%%%%%%%%%%%%%%%%%%%%%%%%%%%%%%%%%%%%%%%%%%%%%%%%%%%%%%%%%%%%%%%%%%%%%%%%%%%%%%%%%%%%%%
and
\begin{align}\label{lem5:eqn2}
 T_3:= \frac{3k}{\big(\bk(z) \big)^2}  \bigg( - \frac{ (1- |z_2|^2)} {(1-|z|^2)^2}  \frac{\partial \bk(z)}{\partial z_2} \cdot  \frac{\partial \bk(z)}{\partial \overline{z_2}} - \frac{ (1- |z_1|^2)} {(1-|z|^2)^2}  \frac{\partial \bk(z)}{\partial z_1} \cdot  \frac{\partial \bk(z)}{\partial \overline{z_1}} + \notag\\[0.15cm] \frac{z_2 \overline{z_1}}{(1-|z|^2)^2} \cdot  \frac{\partial \bk(z)}{\partial z_2} \frac{\partial \bk(z)}{\partial \overline{z_1}} + \frac{z_1 \overline{z_2}}{(1-|z|^2)^2} \frac{\partial \bk(z)}{\partial z_1} \cdot  \frac{\partial \bk(z)}{\partial \overline{z_2}}\bigg) +\notag \\[0.12cm] 
  \frac{1}{\big(\bk(z) \big)^2}  \bigg( \frac{\partial^2 \bk(z)} {\partial z_1 \partial\overline{ z_1}} \cdot \frac{\partial^2 \bk(z)} {\partial z_2 \partial\overline{ z_2}} - \frac{\partial^2 \bk(z)} {\partial z_1 \partial\overline{ z_2}} \cdot \frac{\partial^2 \bk(z)} {\partial z_2 \partial\overline{ z_1}} \bigg);
 \end{align}
%%%%%%%%%%%%%%%%%%%%%%%%%%%%%%%%%%%%%%%%%%%%%%%%%%%%%%%%%%%%%%%%%%%%%%%%%%%%%%%%%%%%%%%%%%%%
and
\begin{align}\label{lem5:eqn3}
T_4:= \frac{1}{\big(\bk(z) \big)^3} \bigg(- \frac{\partial^2 \bk(z)} {\partial z_2 \partial\overline{ z_2}} \cdot  \frac{\partial \bk(z)}{\partial z_1} \cdot  \frac{\partial \bk(z)}{\partial \overline{z_1}} -   \frac{\partial^2 \bk(z)} {\partial z_1 \partial\overline{ z_1}} \cdot  \frac{\partial \bk(z)}{\partial z_2} \cdot  \frac{\partial \bk(z)}{\partial \overline{z_2}} + \notag\\[0.12cm]  \frac{\partial^2 \bk(z)} {\partial z_2 \partial\overline{ z_1}} \cdot  \frac{\partial \bk(z)}{\partial z_2} \cdot  \frac{\partial \bk(z)}{\partial \overline{z_1}}   +\frac{\partial^2 \bk(z)} {\partial z_2 \partial\overline{ z_1}} \cdot  \frac{\partial \bk(z)}{\partial z_2} \cdot  \frac{\partial \bk(z)}{\partial \overline{z_1}}\bigg).
 \end{align}
%%%%%%%%%%%%%%%%%%%%%%%%%%%%%%%%%%%%%%%%%%%%%%%%%%%%%%%%%%%%%%%%%%%%%%%%%%%%%%%%%%%%%%%%%%%%
\end{lem}
%%%%%%%%%%%%%%%%%%%%%%%%%%%%%%%%%%%%%%%%%%%%%%%%%%%%%%%%%%%%%%%%%%%%%%%%%%%%%%%%%%%%%%%%%%%%
\begin{proof}
For any $k\geq 3$ and $z:=(z_1,z_2)^t \in \calf_{\G}$, from equation \eqref{bkmetric2}, we have
\begin{align}\label{lem5eqn1}
\bermet(z)= -\frac{i}{2\pi}\partial \overline{\partial}\log \big(1-|z|^2 \big)^{3k} -\frac{i}{2\pi} \partial \overline{\partial}\log\bk (z).
\end{align}
%%%%%%%%%%%%%%%%%%%%%%%%%%%%%%%%%%%%%%%%%%%%%%%%%%%%%%%%%%%%%%%%%%%%%%%%%%%%%%%%%%%%%%%%%%%%

We compute
\begin{align} \label{lem5eqn2}
&\partial \overline{\partial}\log\bk (z)  = \notag\\& \sum_{i,j=1}^{2} \left( \frac{1}{\bk(z)} \frac{\partial^2 \bk(z)}{\partial z_i \partial\overline{ z_j}} - 
\frac{1}{\big(\bk(z) \big)^2} \frac{\partial \bk(z)}{\partial z_i} \cdot  \frac{\partial \bk(z)}{\partial \overline{z_j}} \right) dz_i \wedge d \overline{z_j}. 
\end{align} 
%%%%%%%%%%%%%%%%%%%%%%%%%%%%%%%%%%%%%%%%%%%%%%%%%%%%%%%%%%%%%%%%%%%%%%%%%%%%%%%%%%%%%%%%%%%%

Combining equations \eqref{lem5eqn1} and \eqref{lem5eqn2}, we derive 
\begin{align}\label{lem5eqn3}
\bervol(z) = \bigg(\frac{i}{2\pi}\bigg)^{2}\det\left( {\begin{smallmatrix} m_{11} & m_{12} \\  &    \\ m_{21}  & m_{22}  \end{smallmatrix}}\right) dz_1 \wedge d\overline{z_1} \wedge dz_2 \wedge d \overline{z_2};
\end{align}
%%%%%%%%%%%%%%%%%%%%%%%%%%%%%%%%%%%%%%%%%%%%%%%%%%%%%%%%%%%%%%%%%%%%%%%%%%%%%%%%%%%%%%%%%%%%
where for $1 \leq i,j \leq 2$, we have 
\begin{align} \label{lem5eqn4}
m_{ii}:=  -\frac{ 3k\big(1- |z_j|^2\big)}{\pi\big(1-|z|^2\big)^2}+ \frac{1}{\bk(z)} \frac{\partial^2 \bk(z)}{\partial z_i \partial\overline{ z_i}} - \frac{1}{\big(\bk(z) \big)^2} \frac{\partial \bk(z)}{\partial z_i} \cdot  \frac{\partial \bk(z)}{\partial \overline{z_i}} 
\end{align}
%%%%%%%%%%%%%%%%%%%%%%%%%%%%%%%%%%%%%%%%%%%%%%%%%%%%%%%%%%%%%%%%%%%%%%%%%%%%%%%%%%%%%%%%%%%%
and for $i\not= j$, we have
\begin{align}\label{lem5eqn5}
 m_{ij}:=  - \frac{3k \,z_j \overline{z_i}}{\pi\big(1-|z|^2\big)^2}  + \frac{1}{\bk(z)} \frac{\partial^2 \bk(z)}{\partial z_i \partial\overline{ z_j}} - \frac{1}{\left(\bk(z) \right)^2} \frac{\partial \bk(z)}{\partial z_i} \cdot  \frac{\partial \bk(z)}{\partial \overline{z_j}}. 
\end{align}
%%%%%%%%%%%%%%%%%%%%%%%%%%%%%%%%%%%%%%%%%%%%%%%%%%%%%%%%%%%%%%%%%%%%%%%%%%%%%%%%%%%%%%%%%%%%

Expanding the determinant described in equation \eqref{lem5eqn3} using equations \eqref{lem5eqn4} and \eqref{lem5eqn5}, and combining it with equation \eqref{defnhypvol}, completes the proof of the lemma. 
\end{proof}
%%%%%%%%%%%%%%%%%%%%%%%%%%%%%%%%%%%%%%%%%%%%%%%%%%%%%%%%%%%%%%%%%%%%%%%%%%%%%%%%%%%%%%%%%%%%

\vspace{0.2cm}
In the ensuing propositions, we estimate the derivatives of the Bergman kernel $\bk$, which enable us to estimate the right hand-side of equation \eqref{lem5:eqn}. 
%%%%%%%%%%%%%%%%%%%%%%%%%%%%%%%%%%%%%%%%%%%%%%%%%%%%%%%%%%%%%%%%%%%%%%%%%%%%%%%%%%%%%%%%%%%%

\begin{prop} \label{prop6}
With notation as above, for any $k \geq 3$ and $z:=(z_1, z_2 )^t \in \calf_{\G}$, we have the following estimate
\begin{align*}
\bigg|\frac{\partial \bk(z)}{\partial z_j} \bigg| \leq  \frac{6k\tilde{C}_{X_{\Gamma},3k} }{(1-|z|^2)^{3k+1}} ,
 \end{align*}    
%%%%%%%%%%%%%%%%%%%%%%%%%%%%%%%%%%%%%%%%%%%%%%%%%%%%%%%%%%%%%%%%%%%%%%%%%%%%%%%%%%%%%%%%%%%%
where $\ctildethreek$ is as defined in equation \eqref{ctildek}.
\end{prop}
%%%%%%%%%%%%%%%%%%%%%%%%%%%%%%%%%%%%%%%%%%%%%%%%%%%%%%%%%%%%%%%%%%%%%%%%%%%%%%%%%%%%%%%%%%%%
\begin{proof}
 For $k \geq 3$, $z:=(z_1, z_2 )^t \in \calf_{\G}$, and $1 \leq j \leq 2$, from equation \eqref{bkseries2}, we derive
\begin{align*}
 \frac{\partial \bk(z)}{\partial z_j} = \sum_{\gamma\in\Gamma}\frac{3k \cthreek \overline{\big(a_{j1}z_1 + a_{j2}z_2+ b_j \big)}}{ \big( \overline{ \big( c_1 z_1+c_2 z_2+ D \big)} -  \overline{\big( a_{11}z_1+ a_{12}z_2 + b_1 \big) } z_1 - \overline{\big( a_{21}z_1+ a_{22}z_2 + b_2 \big) } z_2 \big)^{3k+1}}.
 \end{align*}
%%%%%%%%%%%%%%%%%%%%%%%%%%%%%%%%%%%%%%%%%%%%%%%%%%%%%%%%%%%%%%%%%%%%%%%%%%%%%%%%%%%%%%%%%%%%

Combining the above equation with estimate \eqref{bkineq}, we compute
\begin{align}\label{prop6eqn1}
 (1-|z|^2)^{3k+1}\bigg|\frac{\partial \bk(z)}{\partial z_j} \bigg| \leq  \sum_{\gamma \in \Gamma} \frac{3k \cthreek\big(1-|z|^2\big)^{3k+1}}{\big| Cz+D \big|^{3k} \big|\big\langle \tilde{z}, \tilde{\gamma z} \big\rangle_{\mathrm{hyp}}\big|^{3k+1}}\cdot\bigg| \frac{  a_{j1}z_1 + a_{j2}z_2+ b_j  }{ Cz + D } \bigg|=\notag\\[0.12cm]\sum_{\gamma \in \Gamma}\frac{\big(1-|z|^{2}\big)\cdot\big|(\gamma z )_{j}\big|}{\big|\langle \tilde{z},\tilde{\gamma z}\rangle_{\mathrm{hyp}} \big|}\cdot\frac{3k\cthreek}{\cosh^{3k}\big(\dhyp(z,\gamma z)\slash 2\big)}.
 \end{align}
%%%%%%%%%%%%%%%%%%%%%%%%%%%%%%%%%%%%%%%%%%%%%%%%%%%%%%%%%%%%%%%%%%%%%%%%%%%%%%%%%%%%%%%%%%%%

For any 
\begin{align}\label{gamma:defn}
\gamma:=\left(\begin{smallmatrix} a_{11}&a_{12}&b_1\\a_{21}&a_{22}&b_2\\c_{1}&c_{2}&D \end{smallmatrix}\right)\in \mathrm{SU}\big((2,1),\C \big),
\end{align}
%%%%%%%%%%%%%%%%%%%%%%%%%%%%%%%%%%%%%%%%%%%%%%%%%%%%%%%%%%%%%%%%%%%%%%%%%%%%%%%%%%%%%%%%%%%%
and $z\in \calf_{\G}$, recall that 
\begin{align*}
\gamma z=\big((\gamma z)_1,(\gamma z)_2\big)^{t}=\bigg(\frac{a_{11}z_{1}+a_{12}z_{2}+b_1}{c_1z_1+c_2 z_2+D},\frac{a_{21}z_{1}+a_{22}z_{2}+b_2}{c_1z_1+c_2 z_2+D}\bigg)^{t}
\end{align*}
%%%%%%%%%%%%%%%%%%%%%%%%%%%%%%%%%%%%%%%%%%%%%%%%%%%%%%%%%%%%%%%%%%%%%%%%%%%%%%%%%%%%%%%%%%%%
and
\begin{align}\label{prop6eqn2}
\big| \gamma z\big|^{2}=\big|(\gamma z)_{1}\big|^{2}+\big|(\gamma z)_{2}\big|^{2}\leq 1.
\end{align}
%%%%%%%%%%%%%%%%%%%%%%%%%%%%%%%%%%%%%%%%%%%%%%%%%%%%%%%%%%%%%%%%%%%%%%%%%%%%%%%%%%%%%%%%%%%%

For any $\gamma\in\G$, and $z\in \calf_{\G}$, combining Cauchy-Schwartz inequality with estimate \eqref{prop6eqn2}, we find that
\begin{align}\label{prop6eqn3}
\big|\langle \tilde{z},\tilde{\gamma z}\rangle_{\mathrm{hyp}}\big|=\big|1-z_{1}\overline{(\gamma z)_{1}}-z_{2}\overline{(\gamma z)_{2}}\big|\geq 1-|\gamma z|\cdot|z|\geq 1-|z|.
\end{align}
%%%%%%%%%%%%%%%%%%%%%%%%%%%%%%%%%%%%%%%%%%%%%%%%%%%%%%%%%%%%%%%%%%%%%%%%%%%%%%%%%%%%%%%%%%%%

Combining estimates \eqref{prop6eqn1}, \eqref{prop6eqn2}, and \eqref{prop6eqn3} with Corollary \ref{cor3}, we arrive at the following estimate
\begin{align*}
\big(1-|z|^2\big)^{3k+1}\bigg|\frac{\partial \bk(z)}{\partial z_j} \bigg| \leq     \sum_{\gamma \in \Gamma}\frac{\big(1-|z|^{2}\big)\cdot\big|(\gamma z )_{j}\big|}{\big|\langle \tilde{z},\tilde{\gamma z}\rangle_{\mathrm{hyp}} \big|}\cdot\frac{3k\cthreek}{\cosh^{3k}\big(\dhyp(z,\gamma z)\slash 2\big)}\leq 6k\ctildethreek,
\end{align*}
%%%%%%%%%%%%%%%%%%%%%%%%%%%%%%%%%%%%%%%%%%%%%%%%%%%%%%%%%%%%%%%%%%%%%%%%%%%%%%%%%%%%%%%%%%%%
which completes the proof of the proposition.
\end{proof}
%%%%%%%%%%%%%%%%%%%%%%%%%%%%%%%%%%%%%%%%%%%%%%%%%%%%%%%%%%%%%%%%%%%%%%%%%%%%%%%%%%%%%%%%%%%%

\vspace{0.2cm}
\begin{prop}\label{prop7}
With notation as above, for any $k \geq 3$ and $z:=(z_1, z_2 )^t \in \calf_{\G}$, we have the following estimate
\begin{equation*}
\left|\frac{\partial \bk(z)}{\partial \overline{z_j}} \right|\leq  \frac{6k\,\tilde{C}_{X_{\Gamma},3k}}{\big(1-\big|z\big|^2\big)^{3k+1}} ,
\end{equation*}
%%%%%%%%%%%%%%%%%%%%%%%%%%%%%%%%%%%%%%%%%%%%%%%%%%%%%%%%%%%%%%%%%%%%%%%%%%%%%%%%%%%%%%%%%%%%
where $\ctildethreek$ is as defined in equation \eqref{ctildek}.
\end{prop}
%%%%%%%%%%%%%%%%%%%%%%%%%%%%%%%%%%%%%%%%%%%%%%%%%%%%%%%%%%%%%%%%%%%%%%%%%%%%%%%%%%%%%%%%%%%%
\begin{proof}
For any $k \geq 3 $, and $z:=(z_1, z_2 )^t \in \calf_{\G}$, as $\bk(z)$ is a real-valued function, we find that
\begin{align*}
\bk (z)= \overline{\bk(z)} =  \sum_{\gamma\in\Gamma}\frac{\cthreek}{ \big(( Cz+ D ) -  ( a_{11}z_1+ a_{12}z_2 + b_1 ) \overline{ z_1} - ( a_{21}z_1+ a_{22}z_2 + b_2 ) \overline{ z_2}\big)^{3k}} .
\end{align*}
%%%%%%%%%%%%%%%%%%%%%%%%%%%%%%%%%%%%%%%%%%%%%%%%%%%%%%%%%%%%%%%%%%%%%%%%%%%%%%%%%%%%%%%%%%%%

For $1 \leq j \leq 2 $ , we derive
\begin{align*}
\frac{\partial \bk(z)}{\partial \overline{z_j}} = 
\sum_{\gamma\in\Gamma}\frac{3k \cthreek \big(a_{j1}z_1 + a_{j2}z_2+ b_j \big)}{ \big( ( Cz+ D ) -  \big( a_{11}z_1+ a_{12}z_2 + b_1 \big) \overline{ z_1} - \big( a_{21}z_1+ a_{22}z_2 + b_2 \big)  \overline{z_2}\big)^{3k+1}}.
 \end{align*}
%%%%%%%%%%%%%%%%%%%%%%%%%%%%%%%%%%%%%%%%%%%%%%%%%%%%%%%%%%%%%%%%%%%%%%%%%%%%%%%%%%%%%%%%%%%%

The rest of the proof follows from the same arguments as the ones employed to prove Proposition \eqref{prop6}.       
\end{proof}
%%%%%%%%%%%%%%%%%%%%%%%%%%%%%%%%%%%%%%%%%%%%%%%%%%%%%%%%%%%%%%%%%%%%%%%%%%%%%%%%%%%%%%%%%%%%

\vspace{0.2cm}
\begin{prop}\label{prop8}
With notation as above, for any $k \geq 3$, and $z:=(z_1, z_2 )^t \in \calf_{\G}$, we have the following estimate
\begin{align*}
\bigg| \frac{\partial^2 \bk(z)}{\partial z_i \partial \overline{z_j}} \bigg|  \leq \frac{24k\ctildethreek}{\big(1-|z|^{2}\big)^{3k+2}}
+\frac{12k(3k+1)\ctildethreek}{\big(1-|z|^2\big)^{3k+2}},
\end{align*}
%%%%%%%%%%%%%%%%%%%%%%%%%%%%%%%%%%%%%%%%%%%%%%%%%%%%%%%%%%%%%%%%%%%%%%%%%%%%%%%%%%%%%%%%%%%%
where $\ctildethreek$ is as defined in \eqref{ctildek}.
\end{prop}
%%%%%%%%%%%%%%%%%%%%%%%%%%%%%%%%%%%%%%%%%%%%%%%%%%%%%%%%%%%%%%%%%%%%%%%%%%%%%%%%%%%%%%%%%%%%
\begin{proof}
For any $k \geq 3$, and $z:=(z_1, z_2 )^t \in \calf_{\G}$,  and for any $1\leq i,j \leq 2 $, we compute 
\begin{align}\label{prop8eqn1}
\frac{\partial^2 \bk (z)}{\partial z_i \partial \overline{z_j}}=  
\frac{\partial}{\partial {z_i}}  \sum_{\gamma\in\Gamma}\frac{3k \cthreek (a_{j1}z_1 + a_{j2}z_2+ b_j )}{ \big((Cz+ D) -  ( a_{11}z_1+ a_{12}z_2 + b_1 ) \overline{ z_1}-( a_{21}z_1+ a_{22}z_2 + b_2)  \overline{z_2}\big)^{3k+1}}= \notag\\[0.12cm]\sum_{\gamma \in \Gamma}\frac{3k \cthreek a_{ji} }{ \big( ( Cz+ D ) -  ( a_{11}z_1+ a_{12}z_2 + b_1) \overline{ z_1} - ( a_{21}z_1+ a_{22}z_2 + b_2 )  \overline{z_2}\big)^{3k+1}} - \notag\\[0.12cm]
  \sum_{\gamma \in \Gamma}\frac{3k (3k+1)\cthreek   (a_{j1}z_1+ a_{j2}z_2 + b_j ) ( c_i - a_{1i} \overline{z_1}-a_{2i} \overline{z_2})}{\big(  (Cz+ D ) -  
  ( a_{11}z_1+ a_{12}z_2 + b_1 ) \overline{ z_1} - ( a_{21}z_1+ a_{22}z_2 + b_2 )  \overline{z_2}\big)^{3k+2}}.
\end{align}
%%%%%%%%%%%%%%%%%%%%%%%%%%%%%%%%%%%%%%%%%%%%%%%%%%%%%%%%%%%%%%%%%%%%%%%%%%%%%%%%%%%%%%%%%%%%   

We now estimate each of the two terms on the right hand-side of the above inequality. For the first term, using arguments as in the proof of Proposition \ref{prop6}, we compute
\begin{align}\label{prop8eqn2}  
\sum_{\gamma \in \Gamma}\bigg|\frac{3k \cthreek a_{ji}\,\big(1-\big|z\big|^2\big)^{3k+2} }{ \big(  \big( Cz+ D \big) - \big( a_{11}z_1+ a_{12}z_2 + b_1 \big) \overline{ z_1} - \big( a_{21}z_1+ a_{22}z_2 + b_2 \big)\overline{z_2}\big)^{3k+1}} \bigg|\leq\notag \\[0.1cm]\sum_{\gamma \in \Gamma}\frac{3k  \big|a_{ji}\big|\big(1-\big|z\big|^2\big)^{2}}{\big| Cz+ D \big|\cdot\big|
\langle \tilde{z},\tilde{\gamma z}\rangle_{\mathrm{hyp}}\big|}\cdot\frac{\cthreek}{\cosh^{3k}\big(\dhyp(z,\gamma z)\slash 2\big)}.
\end{align}
%%%%%%%%%%%%%%%%%%%%%%%%%%%%%%%%%%%%%%%%%%%%%%%%%%%%%%%%%%%%%%%%%%%%%%%%%%%%%%%%%%%%%%%%%%%%   

From estimate \eqref{prop6eqn3}, we have the following estimate
\begin{align}\label{prop8eqn3}
\frac{3k  \big|a_{ji}\big|\big(1-\big|z\big|^2\big)^{2}}{\big| Cz+ D \big|\cdot\big|
\langle \tilde{z},\tilde{\gamma z}\rangle_{\mathrm{hyp}}\big|}\leq \frac{6k\,\big|a_{ji}\big|\big(1-\big|z\big|^2\big)}{\big|D\big|\cdot \big(1-\big|\frac{c_1z_1}{D}+\frac{c_2z_2}{D}\big|\big)}.
\end{align}
%%%%%%%%%%%%%%%%%%%%%%%%%%%%%%%%%%%%%%%%%%%%%%%%%%%%%%%%%%%%%%%%%%%%%%%%%%%%%%%%%%%%%%%%%%%%   

As $\gamma \in \Gamma$ with $\gamma$ as in \eqref{gamma:defn}, which is a discreet subgroup of $\mathrm{SU}\big((2,1),\C\big)$, it satisfies the following relation 
\begin{align*}
    \gamma^* H \gamma = H ,\, \,\mathrm{where} \,\,H = \left(\begin{smallmatrix}
1 & 0 & 0 \\
0 & 1 & 0 \\
0 & 0  & -1
\end{smallmatrix}\right)
\end{align*}
%%%%%%%%%%%%%%%%%%%%%%%%%%%%%%%%%%%%%%%%%%%%%%%%%%%%%%%%%%%%%%%%%%%%%%%%%%%%%%%%%%%%%%%%%%%%   

This gives rise to 
\begin{align*} 
    \left(\begin{smallmatrix} \overline{a_{11}} & \overline{a_{21}} & \overline{c_1}\\ \overline{a_{12}}  & \overline{a_{22}}  & \overline{c_2}\\ \overline{b_1}  & \overline{b_2} &  \overline{D} \end{smallmatrix}\right)
\left(\begin{smallmatrix} 1 & 0 & 0 \\
0 & 1 & 0 \\
0 & 0  & -1
    \end{smallmatrix} \right)
    \left(\begin{smallmatrix} a_{11} & a_{12} & b_1\\  a_{21} & a_{22} &  b_2\\ c_1  & c_2 &  D \end{smallmatrix}\right)
    = 
    \left(\begin{smallmatrix}
    1 & 0 & 0 \\
0 & 1 & 0 \\
0 & 0  & -1
    \end{smallmatrix}\right).
\end{align*}
%%%%%%%%%%%%%%%%%%%%%%%%%%%%%%%%%%%%%%%%%%%%%%%%%%%%%%%%%%%%%%%%%%%%%%%%%%%%%%%%%%%%%%%%%%%%

Expanding along the columns, we derive the following relations
\begin{align}\label{matrixrln1}
|a_{1i}|^2 +  |a_{2i}|^2- |c _ i|^2 =1,\,\,\mathrm{for}\,\,i=1,2;\,\,\, |D|^2- |b_{1}|^2 -  |b_{2}|^2=1.
\end{align}
%%%%%%%%%%%%%%%%%%%%%%%%%%%%%%%%%%%%%%%%%%%%%%%%%%%%%%%%%%%%%%%%%%%%%%%%%%%%%%%%%%%%%%%%%%%%

Furthermore, since $\gamma\in \mathrm{SU}\big((2,1),\C\big)$ we find that
\begin{align*}
&\gamma^{\ast}H\gamma=H\implies \gamma H\gamma^{\ast}H\gamma=\gamma H^2=\gamma\implies\\
&\gamma H\gamma^{\ast}H=\mathrm{Id}\implies  \gamma H\gamma^{\ast}=
H\implies \gamma\in \mathrm{SU}\big((2,1),\C\big).
\end{align*}
%%%%%%%%%%%%%%%%%%%%%%%%%%%%%%%%%%%%%%%%%%%%%%%%%%%%%%%%%%%%%%%%%%%%%%%%%%%%%%%%%%%%%%%%%%%%

From similar arguments as the ones employed to prove relations \eqref{matrixrln1}, we derive
\begin{align}\label{matrixrln2}
|a_{i1}|^2 +  |a_{i2}|^2- |b _ i|^2 =1,\,\,\mathrm{for}\,i=1,2;\,\, |D|^2- |c_{1}|^2 -  |c_{2}|^2=1.
\end{align}
%%%%%%%%%%%%%%%%%%%%%%%%%%%%%%%%%%%%%%%%%%%%%%%%%%%%%%%%%%%%%%%%%%%%%%%%%%%%%%%%%%%%%%%%%%%%

So for any $\gamma\in \G$, combining equations \eqref{matrixrln1} and \eqref{matrixrln2}, we infer that
\begin{align*}
\bigg(\frac{c_1}{D},\frac{c_2}{D}\bigg)^t\in \Bt;\,\,|D|\geq 1;\,\,|a_{ji}|\leq 2|D|,\,\,\mathrm{for}\,\,1\leq i,\j\leq 2.
\end{align*}
%%%%%%%%%%%%%%%%%%%%%%%%%%%%%%%%%%%%%%%%%%%%%%%%%%%%%%%%%%%%%%%%%%%%%%%%%%%%%%%%%%%%%%%%%%%%

So using Cauchy-Schwartz inequality, and the fact that $|z|\leq 1$, we derive the following estimate
\begin{align}\label{prop8eqn4}
\frac{6k|a_{ji}|\big(1-|z|^2\big)}{|D|\cdot \big(1-\big|\frac{c_1z_1}{D}+\frac{c_2z_2}{D}\big|\big)}\leq \frac{12k\big(1-|z|^2\big)}{1-|z|}\leq 24k.
\end{align}
%%%%%%%%%%%%%%%%%%%%%%%%%%%%%%%%%%%%%%%%%%%%%%%%%%%%%%%%%%%%%%%%%%%%%%%%%%%%%%%%%%%%%%%%%%%%

Combining estimates \eqref{prop8eqn2}, \eqref{prop8eqn3}, and \eqref{prop8eqn4} with Corollary \ref{cor3}, we derive the following estimate, for the first term on the  right hand-side 
of inequality \eqref{prop8eqn1}
\begin{align}\label{prop8eqn5}
\sum_{\gamma \in \Gamma}\frac{3k \cthreek |a_{ji}|}{ \big|  ( Cz+ D ) - ( a_{11}z_1+ a_{12}z_2 + b_1 ) \overline{ z_1} - ( a_{21}z_1+ a_{22}z_2 + b_2 )\overline{z_2}\big|^{3k+1}}\leq \frac{24k\ctildethreek}{\big(1-|z|^2\big)^{3k+2}}.
\end{align}
%%%%%%%%%%%%%%%%%%%%%%%%%%%%%%%%%%%%%%%%%%%%%%%%%%%%%%%%%%%%%%%%%%%%%%%%%%%%%%%%%%%%%%%%%%%%

For the second term on the right hand-side of inequality \eqref{prop8eqn1}, from arguments as above (estimate \eqref{prop8eqn2}), for any $\gamma \in\G$ with $\gamma$ as in \eqref{gamma:defn}, we observe 
\begin{align*}
&\frac{| a_{j1}z_1+ a_{j2}z_2 + b_j  |\cdot | c_i - a_{1i} \overline{z_1} - a_{2i} \overline{z_2} |}{\big| (Cz+ D ) - ( a_{11}z_1+ a_{12}z_2 + b_1 ) \overline{ z_1} - ( a_{21}z_1+ a_{22}z_2 + b_2 )\overline{z_2}\big|^{2}}=\notag\\[0.12cm]&\frac{\big| (\gamma z)_{j}\big|}{\big| \langle z,\gamma z\rangle_{\mathrm{hyp}}\big|}\cdot 
\frac{| c_i - a_{1i} \overline{z_1} - a_{2i} \overline{z_2} |}{\big|D-b_{1}\overline{z_{1}}-b_{2}\overline{z_{2}}\big|}\cdot\frac{1}{\big| \langle z,\gamma^{\ast} (-z)\rangle_{\mathrm{hyp}}\big|}=\frac{\big| (\gamma z)_{j}\big|}{\big| \langle z,\gamma z\rangle_{\mathrm{hyp}}\big|}\cdot 
\frac{\big| (\gamma^{\ast}(-z))_{i}\big|}{\big| \langle z,\gamma^{\ast} (-z)\rangle_{\mathrm{hyp}}\big|},
 \end{align*}
%%%%%%%%%%%%%%%%%%%%%%%%%%%%%%%%%%%%%%%%%%%%%%%%%%%%%%%%%%%%%%%%%%%%%%%%%%%%%%%%%%%%%%%%%%%%
where $\gamma^{\ast}$ denotes the conjugate transpose of $\gamma$.
%%%%%%%%%%%%%%%%%%%%%%%%%%%%%%%%%%%%%%%%%%%%%%%%%%%%%%%%%%%%%%%%%%%%%%%%%%%%%%%%%%%%%%%%%%%%

\vspace{0.1cm}
Combining estimate \eqref{prop6eqn3} with the observations that $|\gamma z|\leq 1$ and $|\gamma^{\ast}(-z)|\leq 1$, we further derive
\begin{align}\label{prop8eqn6} 
&\frac{\big| a_{j1}z_1+ a_{j2}z_2 + b_j  \big|\cdot\big| c_i - a_{1i} \overline{z_1} - a_{2i} \overline{z_2} \big|}{\big| (Cz+ D ) - ( a_{11}z_1+ a_{12}z_2 + b_1 ) \overline{ z_1} - ( a_{21}z_1+ a_{22}z_2 + b_2 )\overline{z_2}\big|^{2}}\leq\frac{4}{\big(1-|z|^{2}\big)^{2}}.
\end{align}
%%%%%%%%%%%%%%%%%%%%%%%%%%%%%%%%%%%%%%%%%%%%%%%%%%%%%%%%%%%%%%%%%%%%%%%%%%%%%%%%%%%%%%%%%%%%

Combining estimate \eqref{prop8eqn6} with Corollary \ref{cor3}, we derive the following estimate for the second term on the right hand-side of inequality \eqref{prop8eqn1}
\begin{align}\label{prop8eqn7} 
&\sum_{\gamma \in \Gamma}\Bigg|\frac{3k (3k+1)\cthreek   \big(a_{j1}z_1+ a_{j2}z_2 + b_j \big) \big( c_i - a_{1i} \overline{z_1}-a_{2i} \overline{z_2} \big)}{\big(  \big(Cz+ D \big) -  
\big( a_{1pro1}z_1+ a_{12}z_2 + b_1 \big) \overline{ z_1} - \big( a_{21}z_1+ a_{22}z_2 + b_2 \big)  \overline{z_2}\big)^{3k+2}}\Bigg|\notag\leq\\[0.12cm]&
\sum_{\gamma \in \Gamma}\frac{12k(k+1)}{\big(1-|z|^{2} \big)^{3k+2}}\frac{\cthreek}{\cosh^{3k}\big(\dhyp(z,\gamma z)\slash 2\big)}\leq \frac{12k(k+1)\ctildethreek}{\big(1-|z|^{2}\big)^{3k+2}}.
\end{align}
%%%%%%%%%%%%%%%%%%%%%%%%%%%%%%%%%%%%%%%%%%%%%%%%%%%%%%%%%%%%%%%%%%%%%%%%%%%%%%%%%%%%%%%%%%%%

Combining estimates\eqref{prop8eqn1}, \eqref{prop8eqn5}, and \eqref{prop8eqn7}, completes the proof of the proposition. 
\end{proof}
%%%%%%%%%%%%%%%%%%%%%%%%%%%%%%%%%%%%%%%%%%%%%%%%%%%%%%%%%%%%%%%%%%%%%%%%%%%%%%%%%%%%%%%%%%%%

\vspace{0.2cm}
Using Propositions \ref{prop6}  and \ref{prop7}, we now estimate the Bergman metric. 
%%%%%%%%%%%%%%%%%%%%%%%%%%%%%%%%%%%%%%%%%%%%%%%%%%%%%%%%%%%%%%%%%%%%%%%%%%%%%%%%%%%%%%%%%%%%

\begin{thm}\label{thm9}
With the notation as above, for $k \gg 1$ sufficiently large, and $z\in\calf_{\G}$, we have the following estimate
\begin{align*}
\sup_{z\in \calf_{\G}}\bigg| \frac{\bervol(z)}{\hypbnvol(z)} \bigg|\leq 120961 \bigg(\frac{ \ctildethreek}{\cthreek}\bigg)^3k^{4}\log k.
\end{align*}
%%%%%%%%%%%%%%%%%%%%%%%%%%%%%%%%%%%%%%%%%%%%%%%%%%%%%%%%%%%%%%%%%%%%%%%%%%%%%%%%%%%%%%%%%%%%
where $\ctildethreek$ is defined in equation \eqref{ctildek}.
\end{thm}
%%%%%%%%%%%%%%%%%%%%%%%%%%%%%%%%%%%%%%%%%%%%%%%%%%%%%%%%%%%%%%%%%%%%%%%%%%%%%%%%%%%%%%%%%%%%
\begin{proof}
From equation \eqref{lem5:eqn}, we recall
\begin{align}\label{thm9-eqn1}
 \bigg|\frac{\bervol(z)}{\hypbnvol(z)}\bigg| = \frac{\big(1-|z|^2\big)^3}{\pi^2}\big|T_1 + T_2 + T_3 + T_4 \big|,
\end{align}
%%%%%%%%%%%%%%%%%%%%%%%%%%%%%%%%%%%%%%%%%%%%%%%%%%%%%%%%%%%%%%%%%%%%%%%%%%%%%%%%%%%%%%%%%%%%
where $T_1, T_2, T_3$ and $T_4$ are as described in equations \eqref{lem5:eqn1}, \eqref{lem5:eqn2}, and \eqref{lem5:eqn3}. We now derive individual estimates for each terms. Using estimates derived in Propositions \eqref{prop7} and \eqref{prop8}, and using the fact that $k\geq 3$, we deduce 
\begin{align}\label{thm9-eqn2}
 \big|T_1\big|\leq \frac{9k^2}{\pi^2};\quad\big| T_2\big|\leq \frac{432k^2(k+1)\ctildethreek}{\big(1-|z|^{2} \big)\big|\blk(z,z) \big|_{\mathrm{hyp}}};\quad \big|T_4\big|\leq \frac{5184k^3(k+1)\cuctildethreek}{\big(1-|z|^{2} \big)\big|\blk(z,z) \big|^3_{\mathrm{hyp}}}\notag\\[0.12cm] 
 \big|T_3\big|\leq \frac{432k^{3}\sqctildethreek}{\big(1-|z|^{2} \big)\big|\blk(z,z) \big|_{\mathrm{hyp}}^{2}}+\frac{864k^2(k+1)^{2}\sqctildethreek}{\big(1-|z|^{2} \big)\big|\blk(z,z) \big|^2_{\mathrm{hyp}}}. 
 \end{align}
%%%%%%%%%%%%%%%%%%%%%%%%%%%%%%%%%%%%%%%%%%%%%%%%%%%%%%%%%%%%%%%%%%%%%%%%%%%%%%%%%%%%%%%%%%%%

Observe that
\begin{align}\label{thm9-eqn3}
\big|\blk(z,z) \big|_{\mathrm{hyp}}\geq \cthreek-\sum_{\gamma\in\G\backslash \lbrace \mathrm{Id}\rbrace}\frac{\cthreek}{\cosh^{3k}\big( \dhyp(z,\gamma z)\slash 2\big)} .
\end{align}
%%%%%%%%%%%%%%%%%%%%%%%%%%%%%%%%%%%%%%%%%%%%%%%%%%%%%%%%%%%%%%%%%%%%%%%%%%%%%%%%%%%%%%%%%%%%
where $\cthreek$ is as defined in equation \eqref{defncgamma}. For $k\gg1$, from the proof of Corollary \ref{cor3}, it follows that
\begin{align}\label{thm9-eqn4}
\sum_{\gamma\in\G\backslash \lbrace \mathrm{Id}\rbrace}\frac{\cthreek}{\cosh^{3k}\big( \dhyp(z,\gamma z)\slash 2\big)}=O_{X_{\G}}\bigg(\frac{\cthreek}{\cosh^{3k-8}(\rx\slash 2)}\bigg),
\end{align}
%%%%%%%%%%%%%%%%%%%%%%%%%%%%%%%%%%%%%%%%%%%%%%%%%%%%%%%%%%%%%%%%%%%%%%%%%%%%%%%%%%%%%%%%%%%%

For $k\gg 1$, combining estimates \eqref{thm9-eqn3} and \eqref{thm9-eqn4}, we arrive at the following estimate
\begin{align}\label{thm9-eqn5}
\big|\blk(z,z) \big|_{\mathrm{hyp}}\geq \frac{\cthreek}{2}.
\end{align} 
%%%%%%%%%%%%%%%%%%%%%%%%%%%%%%%%%%%%%%%%%%%%%%%%%%%%%%%%%%%%%%%%%%%%%%%%%%%%%%%%%%%%%%%%%%%%

Combining equation \eqref{thm9-eqn1} with estimates \eqref{thm9-eqn2} and \eqref{thm9-eqn5}, and using the fact that $\cthreek\leq \ctildethreek$ we derive
\begin{align}\label{thm9-eqn6}
 \bigg|\frac{\bervol(z)}{\hypbnvol(z)}\bigg|\leq \frac{9k^2}{\pi^2}+\bigg(\frac{ 2\ctildethreek}{\cthreek}\bigg)^3\cdot\frac{15120k^4}{\big( 1-|z|^2\big)}.
\end{align}
%%%%%%%%%%%%%%%%%%%%%%%%%%%%%%%%%%%%%%%%%%%%%%%%%%%%%%%%%%%%%%%%%%%%%%%%%%%%%%%%%%%%%%%%%%%%

For $k\gg1$, as $\calf_{\G}$ is a fixed fundamental domain, and since $X_{\G}$ is compact, we have
\begin{align}\label{thm9-eqn7}
\frac{1}{\big( 1-|z|^2\big)}\leq \log k. 
\end{align}
%%%%%%%%%%%%%%%%%%%%%%%%%%%%%%%%%%%%%%%%%%%%%%%%%%%%%%%%%%%%%%%%%%%%%%%%%%%%%%%%%%%%%%%%%%%%

Combining estimates \eqref{thm9-eqn6} and \eqref{thm9-eqn7}, completes the proof of the theorem.
\end{proof}
%%%%%%%%%%%%%%%%%%%%%%%%%%%%%%%%%%%%%%%%%%%%%%%%%%%%%%%%%%%%%%%%%%%%%%%%%%%%%%%%%%%%%%%%%%%%

\vspace{0.2cm}
\begin{cor}\label{cor10}
With the notation as above, for $k \gg 1$ sufficiently large, and $\epsilon>0$, we have the following estimate
\begin{align*}
\sup_{z\in X_{\G}}\bigg| \frac{\bervol(z)}{\hypbnvol(z)} \bigg|=O_{X_{\G},\epsilon}\big( k^{4+\epsilon}\big),
\end{align*}
%%%%%%%%%%%%%%%%%%%%%%%%%%%%%%%%%%%%%%%%%%%%%%%%%%%%%%%%%%%%%%%%%%%%%%%%%%%%%%%%%%%%%%%%%%%%
where the implied constant depends on the Picard surface $X_{\G}$, and on the choice of $\epsilon>0$.
%%%%%%%%%%%%%%%%%%%%%%%%%%%%%%%%%%%%%%%%%%%%%%%%%%%%%%%%%%%%%%%%%%%%%%%%%%%%%%%%%%%%%%%%%%%%
\begin{proof}
The proof of the corollary follows directly from Theorem \ref{thm9}.
\end{proof}
\end{cor}
%%%%%%%%%%%%%%%%%%%%%%%%%%%%%%%%%%%%%%%%%%%%%%%%%%%%%%%%%%%%%%%%%%%%%%%%%%%%%%%%%%%%%%%%%%%%

\vspace{0.2cm}
\begin{cor}\label{cor11}
With the notation above, for any $z\in \calf_{\G}$, we have the following estimate
\begin{align*}
\lim_{k\rightarrow\infty}\frac{1}{k^4\log k}\bigg| \frac{\bervol(z)}{\hypbnvol(z)} \bigg|\leq 6048.
\end{align*}
%%%%%%%%%%%%%%%%%%%%%%%%%%%%%%%%%%%%%%%%%%%%%%%%%%%%%%%%%%%%%%%%%%%%%%%%%%%%%%%%%%%%%%%%%%%%
\begin{proof}
From estimate \eqref{ctildek}, we have
\begin{align*}
\lim_{k\rightarrow \infty}\frac{ \ctildethreek}{\cthreek}=1. 
\end{align*}
%%%%%%%%%%%%%%%%%%%%%%%%%%%%%%%%%%%%%%%%%%%%%%%%%%%%%%%%%%%%%%%%%%%%%%%%%%%%%%%%%%%%%%%%%%%%
Combining estimate \eqref{thm9-eqn2}  from Theorem \ref{thm9} with the above equation, completes the proof of the corollary. 
\end{proof}
\end{cor}
%%%%%%%%%%%%%%%%%%%%%%%%%%%%%%%%%%%%%%%%%%%%%%%%%%%%%%%%%%%%%%%%%%%%%%%%%%%%%%%%%%%%%%%%%%%%
%%%%%%%%%%%%%%%%%%%%%%%%%%%%%%%%%%%%%%%%%%%%%%%%%%%%%%%%%%%%%%%%%%%%%%%%%%%%%%%%%%%%%%%%%%%%

\vspace{0.2cm}
\subsection*{Acknowledgements}
 The first author acknowledges the support of INSPIRE research grant DST/INSPIRE/04/2015/002263 and the MATRICS grant MTR/2018/000636. 
%%%%%%%%%%%%%%%%%%%%%%%%%%%%%%%%%%%%%%%%%%%%%%%%%%%%%%%%%%%%%%%%%%%%%%%%%%%%%%%%%%%%%%%%%%%%
%%%%%%%%%%%%%%%%%%%%%%%%%%%%%%%%%%%%%%%%%%%%%%%%%%%%%%%%%%%%%%%%%%%%%%%%%%%%%%%%%%%%%%%%%%%%

%%%%%%%%%%%%%%%%%%%%%%%%%%%%%%%%%%%%%%%%%%%%%%%%%%%%%%%%%%%%%%%%%%%%%%%%%%%%%%%%%%%%%%%%%%%%
%%%%%%%%%%%%%%%%%%%%%%%%%%%%%%%%%%%%%%%%%%%%%%%%%%%%%%%%%%%%%%%%%%%%%%%%%%%%%%%%%%%%%%%%%%%%
%%%%%%%%%%%%%%%%%%%%%%%%%%%%%%%%%%%%%%%%%%%%%%%%%%%%%%%%%%%%%%%%%%%%%%%%%%%%%%%%%%%%%%%%%%%%
%%%%%%%%%%%%%%%%%%%%%%%%%%%%%%%%%%%%%%%%%%%%%%%%%%%%%%%%%%%%%%%%%%%%%%%%%%%%%%%%%%%%%%%%%%%%

\end{document}